\documentclass[10pt]{amsart}

\usepackage{amssymb}

\usepackage{palatino}
\usepackage{mathpazo}

\usepackage[alphabetic]{amsrefs}
\usepackage{enumerate}

\usepackage{mathrsfs}

\DeclareMathOperator{\Int}{Int}
\DeclareMathOperator{\Gal}{Gal}

\DeclareMathOperator{\alg}{alg}
\DeclareMathOperator{\un}{un}
\DeclareMathOperator{\sep}{sep}
\DeclareMathOperator{\Lie}{Lie}

\DeclareMathOperator{\Sym}{Sym}
\DeclareMathOperator{\Gr}{Gr}

\DeclareMathOperator{\GL}{GL}
\DeclareMathOperator{\SL}{SL}

\DeclareMathOperator{\End}{End}
\DeclareMathOperator{\Hom}{Hom}

\DeclareMathOperator{\rad}{rad}
\DeclareMathOperator{\ch}{ch}
\DeclareMathOperator{\Ext}{Ext}

\DeclareMathOperator{\Aut}{Aut}
\DeclareMathOperator{\tr}{tr}

\DeclareMathOperator{\aff}{aff}

\DeclareMathOperator{\SP}{Sp}

\DeclareMathOperator{\SU}{SU}
\DeclareMathOperator{\Spec}{Spec}

\DeclareMathOperator{\soc}{soc}

\newcommand{\tensor}{\otimes}
\renewcommand{\setminus}{\ \rule[2.5pt]{7pt}{1pt}\ }

\newcommand{\bc}[2]{#1_{/#2}}

\newcommand{\TT}{\mathcal{T}}

\newcommand{\Z}{\mathbf{Z}}
\newcommand{\R}{\mathbf{R}}

\newcommand{\UU}{\mathcal{U}}

\newcommand{\LL}{\mathscr{L}}

\newcommand{\A}{\mathscr{A}}
\newcommand{\B}{\mathscr{B}}

\newcommand{\G}{\mathbf{G}}

\newcommand{\lie}[1]{\mathfrak{#1}}
\newcommand{\glie}{\lie{g}}

\newcommand{\hlie}{\lie{h}}

\newcommand{\mlie}{\lie{m}}
\newcommand{\rlie}{\lie{r}}

\newcommand{\Kun}{K_{\un}}
\newcommand{\Gun}{G_{/\Kun}}

\newcommand{\Aptun}{A_{\un}}

\newcommand{\Aun}{\A_{\un}}
\newcommand{\Tun}{T_{\un}}

\newcommand{\Ksep}{K_{\sep}}

\newcommand{\ksep}{k_{\sep}}
\newcommand{\kalg}{k_{\alg}}

\newcommand{\mm}{\mathfrak{m}}

\newcommand{\parah}{\mathcal{P}}

\newcommand{\parahun}{\mathcal{P}_{\un}}

\newcommand{\Gm}{{\G_m}}
\newcommand{\Ga}{{\G_a}}
\newcommand{\Aff}{\mathbf{A}}
\newcommand{\Proj}{\mathbf{P}}

\newtheorem*{theorem}{Theorem}

\newtheorem{theoreml}{Theorem}

\newtheorem*{prop}{Proposition}
\newtheorem*{prop-def}{Proposition/Definition}

\swapnumbers

\swapnumbers

\newtheorem{stmt}{}[subsection]

\newtheorem*{defn}{Definition}

\swapnumbers
\theoremstyle{remark}

\newtheorem*{rem}{Remark}

\numberwithin{equation}{subsection}

\DeclareMathOperator{\red}{reduc}
\DeclareMathOperator{\ind}{ind}

\DeclareMathOperator{\triv}{tr}

\newcommand{\e}{\varepsilon}

\newcommand{\Gred}{G_{{\red}}}

\newcommand{\Tred}{T_{{\red}}}

\newcommand{\Build}{\mathscr{I}}

\newcommand{\Facet}{\mathcal{F}}

\begin{document}
\begin{abstract}
  If $G$ is a connected linear algebraic group over the field $k$, a
  Levi factor of $G$ is a reductive complement to the unipotent
  radical of $G$.  If $k$ has positive characteristic, $G$ may have no
  Levi factor, or $G$ may have Levi factors which are not
  geometrically conjugate. We give in this paper some
  \emph{sufficient} conditions for the existence and the conjugacy of
  Levi factors of $G$.

  Let $\A$ be a Henselian discrete valuation ring with fractions $K$
  and with \emph{perfect} residue field $k$ of characteristic $p>0$.
  Let $G$ be a connected and reductive algebraic group over $K$.
  Bruhat and Tits have associated to $G$ certain smooth $\A$-group
  schemes $\parah$ whose generic fibers $\parah_{/K}$ coincide with
  $G$; these are known as \emph{parahoric group schemes}. The special
  fiber $\parah_{/k}$ of a parahoric group scheme is a linear
  algebraic group over $k$. If $G$ splits over an unramified extension
  of $K$, we show that $\parah_{/k}$ has a Levi factor, and that any
  two Levi factors of $\parah_{/k}$ are geometrically conjugate.
\end{abstract}

\title[Levi decompositions]{Levi decompositions of a linear algebraic group}
\author{George J. McNinch}
\address{Department of Mathematics,
         Tufts University,
         503 Boston Avenue,
         Medford, MA 02155,
         USA}
\date{\today}
\email{george.mcninch@tufts.edu, mcninchg@member.ams.org}
\thanks{\noindent Research of McNinch supported
  in part by the US NSA award H98230-08-1-0110.}

\dedicatory{Dedicated to the memory of Vladimir Morozov and to his
  contributions to mathematics}
\maketitle

\setcounter{tocdepth}{1}
\tableofcontents

\section{Introduction}
\label{introduction}

\subsection{Linear groups}
Let $G$ be a connected, linear algebraic group over a field $k$, and
suppose that the unipotent radical of $G_{/\kalg}$ is defined over
$k$; see \S \ref{unipotent-radical-rational}. A \emph{Levi factor} of
$G$ is a complement in $G$ to the unipotent radical.  If $k$ has
characteristic $p >0$, $G$ may have no Levi factors, or $G$ may
non-conjugate Levi factors; see the overview in \S
\ref{levi-factors-overview}.  In this paper, we investigate the
existence and conjugacy of Levi factors. For some previous work on
these matters, see the paper of J.  Humphreys \cite{humphreys}.

Our approach to Levi factors uses results from the cohomology of
linear algebraic groups; see \S \ref{levi-via-cohomology}. For
example, suppose that the unipotent radical $R$ of $G$ is a vector
group which is $G$-equivariantly isomorphic to a \emph{linear
  representation} $V$ of $G$. If $H^2(G,V) = 0$, then $G$ has a Levi
factor, and if $H^1(G,V)=0$ then any two Levi factors of $G$ are
conjugate by an element of $R(k) = V$.  Of course, these results are
well-known for ``abstract'' groups $G$; see e.g. \cite{Brown}*{Ch.
  IV}. We observe here their validity for a linear algebraic group
$G$.

In Theorems \ref{levi-conjugacy} and \ref{easy-cohomology-existence},
we give straightforward extensions of these cohomological criteria for
the existence and conjugacy of Levi factors when $R$ is no longer a
vector group; note that our results require $G$ to satisfy condition
\textbf{(L)} of \S\ref{linearizable}; this condition means that $R$
has a filtration for which each consecutive quotient is a vector group
with a linear action of $\Gred$.

For a reductive group $H$ in characteristic $p$, J. C. Jantzen has
proved that any $H$-module $V$ having $\dim V \le p$ is completely
reducible; see \cite{jantzen-semisimple}. Assume that $\dim R < p$ and
that condition \textbf{(L)} holds for $G$. If $G$ has a Levi factor,
Jantzen's result implies that any two Levi factors of $G$ are
conjugate by an element of $R(k)$.  Under a further assumption on the
character of the $\Gred$-module obtained from $R$, Jantzen's result
implies the existence of a Levi factor of $G$. For these results, see
Theorem \ref{small-modules}.

Finally, if $H$ is reductive and $P$ is a parabolic subgroup, one
knows for any linear representation $V$ of $H$ that $H^i(H,V) \simeq
H^i(P,V)$ for $i \ge 0$. Using this fact, we prove for a group $G$
satisfying \textbf{(L)} that the existence of a Levi factor follows
from the existence of subgroup of $G$ mapping isomorphically to a
Borel subgroup of the reductive quotient $\Gred$; see the results
\ref{initial-criteria} and Theorem \ref{main-existence-result} which
 are crucial to our investigation of Levi factors of the special fibers
of parahoric group schemes undertaken in \S \ref{parahoric-section}.

\subsection{Parahoric group schemes}
We wish to apply the conditions just mentioned to the special fiber of
a parahoric group scheme; we begin our discussion by briefly
describing these group schemes.

Fix a Henselian discrete valuation ring (DVR) $\A$ with fractions $K$
and residues $k$ (recall that if $\A$ is complete, it is Henselian).
We suppose the characteristic $p$ of the residue field $k$ to be
positive.  Moreover, we suppose that $k$ is \emph{perfect}.

To a connected and reductive group $G$ over $K$, Bruhat-Tits associate
a topological space with an action of the group of $K$-rational points
$G(K)$; it is known as the \emph{affine building} $\Build$ of $G$. To
a facet $\Facet$ of an apartment of $\Build$, there corresponds a
smooth affine group scheme $\parah_\Facet$ over $\A$ whose generic
fiber $\parah_{\Facet/K}$ is equal to $G$; this group scheme has the
property that its group of $\A$-points $\parah_\Facet(\A)$ is the
``connected stabilizer'' of $\Facet$ in $G(K)$.

If $G$ is $K$-split and if $\Facet = \{x\}$ consists in a so-called
\emph{special point} $x$ of $\Build$, then the group scheme $\parah_x$
is split and reductive over $\A$. For general $\Facet$, the group
scheme $\parah = \parah_\Facet$ is not reductive; in particular, the
special fiber $\parah_{/k}$ -- the linear algebraic group
obtained from $\parah$ by base-change to $\Spec(k)$ -- need not be
reductive over $k$.

In his article in the 1977 Corvallis proceedings
\cite{corvallis}*{\S3.5}, J.  Tits wrote that the linear group
$\parah_{/k}$ possesses a Levi factor.  Since the article was intended
-- as Tits wrote in its opening paragraphs -- for utilizers, proofs
were mostly omitted. In particular, no proof of this statement was
given in \emph{loc. cit.}, and none seems to have appeared elsewhere.

We apply the results of \S\ref{Levi-criteria} to obtain the following
results:
\begin{theoreml}
  Let $\parah$ be a parahoric group scheme over $\A$ with generic
  fiber $G = G_{/K}$.
  \begin{enumerate}[(i)]
  \item If $G$ is split over $K$ and if $S$ is a maximal split torus
    of $\parah_{/k}$, then $\parah_{/k}$ has a unique Levi factor
    containing $S$. In particular, any two Levi factors of
    $\parah_{/k}$ are $\parah(k)$-conjugate.
  \item If $G_{/L}$ is split for an unramified extension $K \subset
    L$, then $\parah_{/k}$ has a Levi factor, and any two
    Levi factors of $\parah_{/k}$ are geometrically conjugate.
  \end{enumerate}
\end{theoreml}

Recall that subgroups (or elements, or ...) are \emph{geometrically
  conjugate} under the action of a linear algebraic group $H$ over $k$
provided they are conjugate by an element of $H(\kalg)$ where $\kalg$
is an algebraic closure of $k$.

We observe that if $G_{/L}$ is not split for any unramified extension
$L \supset K$, the special fiber $\parah_{/k}$ of a parahoric group
scheme can possess Levi factors which are not geometrically conjugate;
see the example given in \S\ref{unitary-group-example}

After seeing a preliminary version of this manuscript, Gopal Prasad
explained to the author that Rousseau's theorem (\cite{rousseau}
and \cite{prasad-galois-fixed}) should allow one to prove the
\emph{existence} of a Levi factor in the special fiber of a parahoric
group scheme under the weaker assumption that $G$ splits over a
\emph{tamely ramified} extension of $K$; this result will appear in a
subsequent paper.

\subsection{Terminology}
\label{terminology}

By a linear algebraic group $G$ over a field $k$ we mean a smooth,
affine group scheme of finite type over $k$. When we speak of a closed
subgroup of an algebraic group $G$, we mean a closed subgroup scheme
over $k$; thus the subgroup is required to be ``defined over $k$'' in
the language of \cite{springer-LAG} or \cite{borel-LAG}. Similar
remarks apply to homomorphisms between linear algebraic groups. Note
that we will occasionally use the terminology ``$k$-subgroup'' or
``$k$-homomorphism'' for emphasis.  If $\ell \supset k$ is a field
extension, we write $G_{/\ell}$ for the linear algebraic group over
$\ell$ obtained by extension of scalars.  If $g\in G(k)$ is a
$k$-rational point of the linear algebraic group $G$, we write
$\Int(g)$ for the inner automorphism of $G$ determined by $g$.

Following terminology in \cite{serre-AGCF}*{\S VII.1}, we will say
that the sequence 
\begin{equation*}
  1 \to N \xrightarrow{i} G \xrightarrow{\pi} H \to 1
\end{equation*}
of linear algebraic groups is \emph{strictly exact} if the sequence
\begin{equation*}
  1 \to N(\kalg) \xrightarrow{i} G(\kalg) \xrightarrow{\pi} H(\kalg) \to 1
\end{equation*}
is exact for an algebraic closure $\kalg$ of $k$, and if $i$ and
$\pi$ are \emph{separable} so that $\pi$ induces an isomorphism $G/N
\to H$ of algebraic groups.

\subsection{Acknowledgments}
The author would like to thank  B. Conrad, J. Humphreys, J. C.
Jantzen, and an anonymous referee for useful remarks on a preliminary
version of this manuscript. In particular, I thank the referee for
communicating the  proof of Proposition \ref{parabolic} given
here; it is much simpler than the one I originally gave.

\section{Some preliminaries}

Throughout this section, $G$ denotes a linear algebraic group over the
field $k$; we make further assumptions on $G$ and on $k$ in the
final paragraph \S \ref{representation-theory}.

\subsection{The question of rationality of the unipotent radical over a ground field}
\label{unipotent-radical-rational}
The unipotent radical of $G$ is the largest connected, unipotent,
normal subgroup $R = R_uG_{/\kalg}$ of $G_{/\kalg}$, where $\kalg$ is an
algebraic closure of $k$.  As the following example shows, the
subgroup $R$ is  in general not defined over $k$.

\begin{stmt}
  \label{stmt:pseudo-red} Let $\ell$ be a purely inseparable extension
  of $k$ having degree $p$; say $\ell = k(a)$ with $a^p \in k$.
  Consider the $k$-group $G = R_{\ell/k} \Gm$ obtained from the
  multiplicative group $\Gm$ by restriction of scalars.  The unipotent
  radical of the $\ell$-group $G_{/\ell}$ is defined over $\ell$ and
  has dimension $p-1$, but any normal, connected, unipotent, smooth
  subgroup scheme of $G$ defined over $k$ is trivial.  
\end{stmt}

\begin{proof}
  See \cite{conrad-gabber-prasad}*{Example 1.1.3}, and see e.g.
  \cite{conrad-gabber-prasad}*{\S A.5} or \cite{springer-LAG}*{11.4}
  for the notion of ``restriction of scalars''.
\end{proof}

\noindent In what follows, we will often consider the following assumption on
$G$.
\begin{itemize}
\item[\textbf{(R)}] The unipotent radical $R = R_uG$ is defined over
  $k$.
\end{itemize}
Note the following:
\begin{stmt}
  If $k$ is perfect, then condition \textbf{(R)} holds for the linear
  group $G$; i.e. $R$ is a $k$-subgroup of $G$.
\end{stmt}

\begin{proof}
  This follows by Galois descent; see \cite{springer-LAG}*{11.2.8, 14.4.5(v)}.
\end{proof}

\subsection{Groups with a normal, split unipotent subgroup}
\label{complement-to-R}

Recall \cite{springer-LAG}*{\S14} that a connected, unipotent linear
group $U$ is said to be $k$-split provided that there is a sequence
\begin{equation*}
  1 = U_m \subset U_{m-1} \subset \cdots \subset U_1 \subset U_0 = U
\end{equation*}
of closed, connected, normal subgroups of $U$ such that each quotient
$U_i/U_{i+1}$ is isomorphic to $\G_{a/k}$.  

\begin{stmt}[\cite{springer-LAG}*{14.3.10}]
  If $k$ is perfect, every connected unipotent group over $k$
  is $k$-split
\end{stmt}

\begin{defn}
  A \textbf{vector group} $U$ is a connected, split, abelian, unipotent
  group having  exponent $p$ in case $k$ has characteristic $p>0$.
\end{defn}

\begin{stmt}
  \label{vectorgroup}
  The linear algebraic group $U$ is a vector group if and only if it is
  isomorphic to a direct product of copies of $\G_{a/k}$.
\end{stmt}
\begin{proof}
  \cite{oesterle}*{V.2.3}.
\end{proof}

Suppose now that $R \subset G$ is a normal, connected, 
unipotent subgroup of $G$.  By a \emph{splitting sequence for $R$} we
mean a sequence
\begin{equation*}
  1 = R_m \subset R_{m-1} \subset \cdots \subset R_1 \subset R_0 = R
\end{equation*}
of closed, connected, $k$-subgroups of $R$ which are
normal in $G$ with the property that each quotient
$R_i/R_{i+1}$ is a vector group.

If $R$ possesses a splitting sequence, evidently $R$ is $k$-split. Conversely:
\begin{stmt}
  \label{split-radical}
  If $R$ is $k$-split, then $R$ has a splitting sequence.
\end{stmt}

\begin{proof}
  Consider first the lower central series where for $i>0$, $R_i =
  (R,R_{i-1})$. Since $R$ is split, it follows from
  \cite{springer-LAG}*{14.3.12(3),(2)} that each $R_i$ is $k$-split,
  and that each quotient $R_i/R_{i+1}$ is a $k$-split abelian
  unipotent group.  Thus we are reduced to the case where $R$ is
  abelian $k$-split unipotent.  If the characteristic $p$ of $k$ is 0
  there is nothing left to do, so suppose $p>0$. If $R$ has exponent
  $p$, it is a vector group.  The result for any commutative $R$ now
  follows by induction on the exponent by considering the filtration
  $1 \to R^p \to R$; note that $R^p$ is $k$-split by
  \cite{springer-LAG}*{14.3.12(2)}.
\end{proof}

\begin{prop}
  Suppose that the linear algebraic group $G$ has a normal, connected, $k$-split
  unipotent subgroup $R$. Then there is a morphism of $k$-varieties
  $s:G/R \to G$ which is a section to $\pi$.  In particular,
  $\pi:G(k) \to (G/R)(k)$ is surjective.
\end{prop}

\begin{proof}
  This is a consequence of results of Rosenlicht
  \cite{rosenlicht-rationality}; Lemma 2 of \emph{loc. cit.} shows that
  $\pi:G \to G/R$ is a Zariski-locally trivial principal
  $R$-bundle, and Theorem 1 of \emph{loc. cit.} shows, since $G/R$
  is affine, that $\pi$ has a regular cross-section defined over $k$.
\end{proof}

\subsection{Condition \textbf{(L)}}
\label{linearizable}

Let $U$ be a vector group, and suppose that  $G$ acts by group
automorphisms on $U$.
\begin{defn}
  The action of $G$ on $U$ will be said to be \textbf{linearizable} if
  there is a $G$-equivariant isomorphism between $U$ and the Lie
  algebra $\Lie(U)$, where $\Lie(U)$ is viewed as a vector group with
  its natural action of $G$.
\end{defn}
The next result is straightforward:
\begin{stmt}
  The following are equivalent:
  \begin{enumerate}[(a)]
  \item The action of $G$ on $U$ is linearizable.
  \item There is a $G$-equivariant isomorphism $U \simeq V$ for
    \emph{some} linear representation $V$ of $G$.
  \item There is an action of the 1 dimensional split torus $\Gm$ on
    $U$ by $G$-equivariant group automorphisms such that the resulting
    action on $\Lie(U)$ is scalar multiplication.
  \end{enumerate}
\end{stmt}

Suppose that condition \textbf{(R)} holds for the linear algebraic group $G$,
and that the unipotent radical $R$ of $G$ is $k$-split.  For any
splitting sequence of $R$ as in \ref{split-radical}, notice that the
conjugation action of $R$ on the quotient $R_i/R_{i+1}$ is trivial for
each $i$; i.e. the reductive quotient $\Gred = G/R$ acts on each
vector group $R_i/R_{i+1}$.

We will be interested in the following condition on $G$:
\begin{enumerate}
\item[\textbf{(L)}] Condition \textbf{(R)} holds for $G$, the
  unipotent radical $R$ is $k$-split, and for some splitting
  sequence of $R$ as in \ref{split-radical}, the
  $\Gred$-action on the vector group $R_i/R_{i+1}$ is linearizable for
  each $i$.
\end{enumerate}
A splitting sequence $1 = R_n \subset R_{n-1} \subset \cdots \subset R_0 =
R$ for which \textbf{(L)} holds will be called a \textbf{linearizable
splitting sequence} for $R$.

\begin{rem}
  \begin{enumerate}[(a)]
  \item If $P$ is a parabolic subgroup of a reductive group $G$, then
    it is well-known that condition \textbf{(L)} holds for $P$.
  \item Suppose that $G$ is \emph{reductive} and that $G$ acts by
    group automorphisms on the vector group $U$. The author is aware
    neither of a proof that the $G$-action on $U$ is always
    linearizable, nor of a counterexample to that statement.
  \item The proof of \ref{split-radical} amounts to a construction of
    a \emph{canonical} splitting sequence for which the $R_i$ are
    \emph{characteristic} in $G$. Nevertheless, we only require the
    condition in \textbf{(L)} to hold for \emph{some} splitting
    sequence; this raises the question: if $k \subset \ell$ is a
    separable field extension and \textbf{(L)} holds for $G_{/\ell}$,
    does \textbf{(L)} hold for $G$?
  \end{enumerate}
\end{rem}

\subsection{Recollections: representations of a reductive group}
\label{representation-theory}
In this section, let $k$ be algebraically closed and let $G$ be a
connected and reductive algebraic group over $k$. Fix a maximal torus
$T$ of $G$ and a Borel subgroup $B$ containing $T$.

A dominant weight $\tau \in X^*(T)$, determines the \emph{standard
  module} $H^0(\tau) = \ind_B^G(\tau)$ \cite{JRAG}*{II.2.1} which can
be realized as the global sections of a suitable $G$-linearized line
bundle on $G/B$.  The \emph{Weyl module} $V(\tau)$ is the
contragredient $H^0(-w_0\tau)^\vee$ \cite{JRAG}*{II.2.13} (where $w_0$
is the long word in the Weyl group $N_G(T)/T$).

There is a simple module $L(\tau)$ for which 
\begin{equation*}
  L(\tau) \simeq \soc  H^0(\tau) \quad \text{and} 
  \quad  L(\tau) \simeq V(\tau)/\rad V(\tau),
\end{equation*}
where $\soc(-)$ denotes the largest semisimple submodule, and
$\rad(-)$ denotes the largest submodule for which the quotient is
semisimple; see \cite{JRAG}*{II.2.3, II.2.14}. Moreover, if $L$ is a
simple $G$-module, then $L$ is isomorphic to $L(\tau)$ for a unique
dominant weight $\tau$ \cite{JRAG}*{Prop. II.2.4}. 

We record the following:
\begin{stmt}[\cite{JRAG}*{Prop. II.2.14}]
  \label{ext-by-Weyl-radical}
  If $\tau,\gamma \in X^*(T)$ are dominant weights for which $\gamma \not > \tau$,
  then 
  \begin{equation*}
    \Ext^1_G(L(\tau),L(\gamma)) \simeq \Hom_G(\rad V(\tau),L(\gamma)).
  \end{equation*}
\end{stmt}

\begin{stmt}[\cite{JRAG}*{Prop. II.4.16}]
  \label{standard-Weyl-ext}
  If $\tau,\gamma \in X^*(T)$  are dominant weights, then
  \begin{equation*}
    \Ext^i(V(\tau),H^0(\gamma)) = 0 \quad \text{for all $i>0$}.
  \end{equation*}
  In particular, $H^i(G,H^0(\gamma)) = 0$ for all $i > 0$.
\end{stmt}

For a $G$-module $V$, we may write $V = \bigoplus_{\lambda \in X^*(T)} V_\lambda$
as a direct sum of its weight spaces; the character of $V$ is then 
\begin{equation*}
  \ch V = \sum_{\lambda \in X^*(T)} \dim V_\lambda e^\lambda
\end{equation*}
viewed as an element of the integral group ring $\Z[X^*(T)]$ (which
has $\Z$-basis consisting of the $e^\lambda$). For a dominant weight
$\lambda$ we write $\chi(\lambda) = \ch H^0(\lambda) = \ch
V(\lambda)$; note that $\chi(\lambda)$ is given by Weyl's character
formula \cite{JRAG}*{II.5.10}, so in particular $\dim H^0(\lambda) =
\dim V(\lambda)$ is given by the Weyl degree formula.

We record:
\begin{stmt}[\cite{JRAG}*{II.5.8}]
  \label{character-bases}
  The set $\{\ch L \mid \text{$L$ a simple $G$-module}\}$ is a
  $\Z$-basis for $\Z[X^*(T)]^W$.
\end{stmt}

\section{Levi decompositions}
\label{levi-factors-overview}

In this section, we discuss the existence and conjugacy of Levi
factors in linear algebraic groups over a field $k$. We will write $G$ for a
\emph{connected} linear algebraic group defined over $k$.

\subsection{Definitions and first remarks}
Assume that \textbf{(R)} holds for  $G$.
\begin{defn}
  A \textbf{Levi factor} of $G$ is a closed $k$-subgroup $M$ of $G$ such
  that the product mapping
  \begin{equation*}
    (x,y) \mapsto xy:M \ltimes R \to G 
  \end{equation*}
  is a $k$-isomorphism of algebraic groups, where $M \ltimes R$ denotes
  the \emph{semidirect product} (for the action of $M$ on $R$ by
  conjugation).
\end{defn}
If there is such a Levi factor $M$, one says that $G$ has a \emph{Levi
  decomposition} over $k$. If $M$ is a Levi factor of $G$, then the
projection mapping $G \to G/R$ induces an isomorphism $M
\xrightarrow{\sim} G/R$ so that $M$ is connected and reductive.

Suppose for a moment that $k$ has characteristic zero. Apply Levi's
theorem \cite{Jacobson}*{III.9} to the finite dimensional Lie algebra
$\glie = \Lie(G)$ to find a semisimple Lie subalgebra $\mlie \subset
\glie$ such that $\glie = \mlie \oplus \rlie$ where $\rlie$ is the
radical of $\glie$. Now, $[\mlie,\mlie] = \mlie$, so that $\mlie$ is
an \emph{algebraic Lie subalgebra} by \cite{borel-LAG}*{7.9}. This
condition means that there is a closed connected subgroup $M \subset
G$ with $\Lie(M) = \mlie$; evidently, $M$ is semisimple. Choosing a
maximal torus $T_0$ of $M$ and a maximal torus $T$ of $G$ containing
$T_0$, one finds that $M.T$ is a reductive subgroup of $G$ which is a
complement to the unipotent radical $R$. In summary: linear algebraic groups in
characteristic zero always have a Levi decomposition. Moreover, it
follows from Theorem \ref{levi-conjugacy} below that any two Levi
factors of $G$ are conjugate by an element of the unipotent radical.

On the other hand, if $k$ has positive characteristic, the situation
is more complicated.
\begin{itemize}
\item $G$ need not have a Levi factor; see \S
  \ref{sub:levi-failure-Witt} and \S
  \ref{sub:levi-failure-cohomological}.
\item Two Levi factors for $G$ need not be geometrically conjugate;
  see the example given by Borel and Tits in \cite{borel-tits}*{3.15}
  and see \S\ref{unitary-group-example}.
\item Even in the case of an algebraically closed field $k$, in order
  that a subgroup $M \subset G$ be a Levi factor, it is crucial that
  multiplication determine an isomorphism $M \ltimes R \to G$ of
  \emph{algebraic groups} and not merely of the ``abstract'' group of
  points; see the example in \S\ref{abstract-insufficient}.
\end{itemize}

\subsection{Linear algebraic groups with no Levi decomposition: Witt vector construction}
\label{sub:levi-failure-Witt}
Suppose that $k$ is a perfect field of characteristic $p>0$, let
$W(k)$ be the ring of Witt vectors over $k$, and let $W_2(k) =
W(k)/p^2W(k)$ be the ring of length-two Witt vectors
\cite{Serre-LocalFields}*{II.\S6}. One may view $W_2(k)$ as a
\emph{ring-variety} over $k$; as a variety, $W_2(k)$ is just affine
2-space $\bc{\Aff^2}{k}$.

Let $G_{/W_2(k)}$ be a split semisimple group scheme over $W_2(k)$, and let
$G_{/k}$ be the corresponding semisimple algebraic group obtained by
$G_{/k} = G_{/W_2(k)} \tensor_{W_2(k)} k$. Using the point-of-view found in
\cite{conrad-gabber-prasad}*{A.6} -- to which we refer for details
concerning the following sketch --, we may view the group $G(W_2(k))$
as a linear algebraic group over $k$ of dimension $2\cdot \dim G_{/k}$. To
do this, we note that if $X$ is an affine $W_2(k)$-scheme of finite
type, the functor on $k$-algebras
\begin{equation*}
  \Lambda \mapsto X(W_2(\Lambda))
\end{equation*}
is represented by an affine $k$-scheme $X_2$ of finite type; the
assignment $X \mapsto X_2$ is known as the \emph{Greenberg functor}.

Let $H$ be the value of the Greenberg functor at the smooth affine
$W_2(k)$-group scheme $G_{/W_2(k)}$.  The natural ring homomorphism
$W_2 \to k$ determines a surjective mapping of algebraic groups $H \to
G_{/k}$. According to \cite{conrad-gabber-prasad}*{Lemma A.6.2}, the
kernel of the mapping $H \to G_{/k}$ is a vector group which identifies
as a representation for $G_{/k}$ with the Frobenius twist $\glie^{[1]}$ of the 
adjoint representation $\glie$.
\begin{stmt}
  \label{stmt:Witt-non-split} 
  \cite{conrad-gabber-prasad}*{Prop. A.6.4}
  The group $H$ has reductive quotient $G_{/k}$, and $H$
  has no Levi decomposition.
\end{stmt}

\begin{rem}
  \begin{enumerate}[(a)]
  \item In case $G = \SL_2$, one can see also
    \cite{mcninch-witt}*{Remark 5} for a different argument that
    $\SL_2(W_2)$ has no Levi decomposition.
  \item In fact, the example of the 6-dimensional linear group
    $\SL_2(W_2)$ with reductive quotient $\SL_{2/k}$ and no Levi
    factor was communicated in 1967 by J. Tits (who learned it from P.
    Roquette) in a letter to J.  Humphreys in connection with
    Humphreys' paper \cite{humphreys}.
  \end{enumerate}
\end{rem}

\subsection{Complements to the unipotent radical as an abstract group are insufficient}
\label{abstract-insufficient}
We first recall the following:
\begin{stmt}[\cite{borel-LAG}*{14.19}]
  \label{stmt:Levi-of-parabolic}
  Let $P$ be a parabolic subgroup of a reductive group $G$.  Then $P$
  has a Levi factor defined over $k$, and any two Levi factors of
  $P$ defined over $k$ are conjugate by a unique element of
  $R_u(P)(k)$.
\end{stmt}

In defining a Levi factor, the requirement that the product mapping $M
\ltimes R \to G$ be an isomorphism \emph{of algebraic groups} is
essential. Write $\pi:G \to G/R =\Gred$ for the quotient mapping, and
note that if $M$ is a Levi factor, then $\pi_{\mid M}$ is an
\emph{isomorphism of algebraic groups} (see \ref{split-condition}
below).  The following example shows that a linear algebraic group may
have a subgroup $M'$ for which $\pi_{\mid M'}$ is a purely inseparable
isogeny $M' \to \Gred$ -- and in particular determines an isomorphism
of ``abstract'' groups $M'(\kalg) \to\Gred(\kalg)$ for an algebraic
closure $\kalg$ of $k$ -- but $M'$ is not a Levi factor of $G$.

Let $W$ be a $d$ dimensional $k$-vector space for some $d \ge 2$ where
$k$ is an algebraically closed field of characteristic $p > 0$, and
let $V = \Sym^p W$ be the $p$-th symmetric power of $W$.  Consider the
subspace $W^{[1]}$ of $V$ consisting of the $p$-th powers $w^p$ of all
vectors $w \in W$.  The stabilizer $P$ in $\GL(V)$ of the point
determined by $W^{[1]}$ in the Grassmann variety $\Gr_d(V)$ is a
(maximal) parabolic subgroup of $\GL(V)$.

Choose any linear complement $W'$ to $W^{[1]}$ in $V$ and write $M'$
for the reductive subgroup of $P$ generated by the image of $\GL(W)$
under its representation on $V$ together with the subgroup $\GL(W')$
acting as the identity on $W^{[1]}$.

\begin{stmt}
  Write $P_{\red}$ for the reductive quotient of $P$. Then $\pi_{\mid
    M'}:M' \to P_{\red}$ is a purely inseparable isogeny. However, $M'$ is
  not a Levi factor of $P$.
\end{stmt}

\begin{proof}[Sketch]
  Since $M'(k) \cap R_uP(k)$ is trivial, the group of points $M'(k)$
  maps isomorphically to the group of $k$-points of the reductive
  quotient
  \begin{equation*}
    (P/R_uP)(k) \simeq
    \GL(W^{[1]})(k) \times \GL(V/W^{[1]})(k).
  \end{equation*}
  Thus indeed $\pi_{\mid M'}$ is a purely inseparable isogeny.

  Viewing $V$ as a representation for $\GL(W)$, the subspace $W^{[1]}$
  is $\GL(W)$-invariant; the reader may verify that the exact sequence
  of $\GL(W)$-modules
  \begin{equation*}
    0 \to W^{[1]} \to V \to V/W^{[1]} \to 0
  \end{equation*}
  is \emph{not split exact.} In fact, as a representation for
  $\GL(W)$, $V$ identifies with the \emph{standard} $\GL(W)$-module
  $H^0(\Proj(W),\LL^{\tensor p})$ of global sections of the line
  bundle $\LL^{\tensor p}$ on the projective space $\Proj(W)$, where
  $\LL$ is the $\GL(W)$-linearized line bundle whose global sections
  $H^0(\Proj(W),\LL)$ afford the natural $\GL(W)$-module $W$; standard
  modules for $\GL(W)$ are known to be \emph{indecomposable} with
  simple socle \cite{JRAG}*{II.2.3}; in this case $\soc_{\GL(W)}(V) =
  W^{[1]}$ is the first Frobenius twist of the natural module $W$.

  The ``standard'' Levi factor $M = \GL(W^{[1]}) \times \GL(W')$ of
  $P$ leaves invariant a linear complement (namely $W'$) to $W^{[1]}$
  in $V$.  Since there is no $M'$-stable complement to $W^{[1]}$ in
  $V$, $M'$ is not conjugate to the Levi factor $M$. In particular,
  \ref{stmt:Levi-of-parabolic} shows that $M'$ is not a Levi subgroup
  of $P$. (It is in fact quite easy to check that $d\pi_{\mid M'}$ is not an
  isomorphism --  the kernel of this map coincides with the Lie
  algebra of the image of $\GL(W)$.)
\end{proof}

\section{Cohomology and Levi factors}
\label{levi-via-cohomology}

\subsection{The Hochschild complex}
Let $G$ be a linear algebraic group over $k$.  The functor $V \mapsto
H^\bullet(G,V)$ which assigns to $V$ the cohomology groups $H^n(G,V)$
for $n \ge 0$ is the derived functor of the fixed point functor $V
\mapsto V^G$; cf. \cite{JRAG}*{\S I.4}.

Fix a representation  $V$ of $G$. 
\begin{stmt}
  \label{cohomology-base-change} 
  For any commutative $k$-algebra $\Lambda$ and any $n \ge 0$, the
  natural mapping 
  \begin{equation*}
    H^n(G,V) \tensor_k \Lambda \xrightarrow{\sim}
    H^n(G_{/\Lambda},V \tensor_k \Lambda) 
  \end{equation*}
  is an isomorphism.
\end{stmt}

\begin{proof}
  Since $k$ is a field, $V$ is flat over $k$ and the assertion follows
  from \cite{JRAG}*{Prop. I.4.18}.
\end{proof}

The cohomology $H^\bullet(G,V)$ can be computed using the
\emph{Hochschild complex} $C^\bullet(G,V)$; cf.  \cite{JRAG}*{I.4.14}.
Write $C^\bullet = C^\bullet(G,V)$. Then $C^0 = V$, and $C^n =
k[G]^{\tensor n} \tensor_k V$ for $n \ge 1$. When $V$ is finite
dimensional, $V$ may be viewed as an algebraic variety over $k$ and then
$C^n$ is the collection of regular functions $\prod^n G \to V$. The
boundary mappings $\partial^n = \partial^n_V:C^n \to C^{n+1}$ are
described in \cite{JRAG}*{I.4.14}
\begin{stmt}\cite{JRAG}*{Prop I.4.16}
  The group cohomology $H^\bullet(G,V)$ is equal to the cohomology of the
  complex $(C^\bullet(G,V),\partial^\bullet)$.
\end{stmt}

In some arguments to follow, we shall need slightly more than the
Hochschild complex itself.  Following the proof of \cite{JRAG}*{Prop
  I.4.16} one finds:
\begin{stmt}
  \label{injective-res-that-gives-Hochschild}
  There is an injective resolution $V \to I^\bullet(G,V)$ of $V$ as
  $G$-module such that
  \begin{equation*}
    C^\bullet(G,V) = (I^\bullet(G,V))^G.
  \end{equation*}
\end{stmt}

\begin{proof}
  Indeed, as in \cite{JRAG}*{\S I.4.15} one takes
  \begin{equation*}
    I^\bullet(G,V) = C^\bullet(G,k[G]) \tensor V_{\triv}
  \end{equation*}
  where $V_{\triv}$ denotes $V$ with \emph{trivial} $G$-action.  The
  $G$-module structure on
  \begin{equation*}
    I^n(G,V) = k[G]^{\tensor(n+1)} \tensor  V_{\triv}
  \end{equation*}
  is given by the regular representation of $G$ on the \emph{first}
  tensor factor; $G$ acts trivially on all other tensor factors.  The
  boundary mapping $I^n(G,V) \to I^{n+1}(G,V)$ is simply
  $\partial_{k[G]} \tensor 1_V$.

  Now compose the natural mapping $(v \mapsto 1 \tensor v):V \to k[G]
  \tensor V$ with the $G$-isomorphism
  \begin{equation*}
    k[G]\tensor V \xrightarrow{\sim} k[G] \tensor V_{\operatorname{tr}}
  \end{equation*}  
  of \cite{JRAG}*{I.3.7(3)} in order to find the required $G$-module
  mapping $V \to I^0(G,V)$.

  It is proved in \cite{JRAG}*{\S 1.4.15} that $0 \to V \to
  I^\bullet(G,V)$ is exact, and it is clear by construction that the
  complex $C^\bullet(G,V)$ identifies with $(I^\bullet(G,V))^G$.
\end{proof}

\subsection{Parabolic subgroups}
\label{parabolic}

If $H \subset G$ is a (closed) subgroup, and if $V$ is a
representation of $G$, then the co-morphism $j^*:k[G] \to k[H]$ to the
inclusion $j:H \to G$ yields a chain mapping
\begin{equation*}
  \rho:C^\bullet(G,V) \to C^\bullet(H,V) \quad 
\end{equation*}
by the rule $\rho^n = (j^*)^{\tensor n} \tensor 1:k[G]^{\otimes n}
\tensor V \to k[H]^{\otimes n} \tensor V$.  In particular, if $V$ is
finite dimensional and $f \in C^n(G,V)$ is viewed as a regular
function $f:G^n \to V$, then $\rho(f)$ is just the restriction
$f_{\mid H^n}:H^n \to V$.

Moreover,
$\widetilde \rho = \rho \tensor 1_{V_{\triv}}$ determines a chain
mapping
\begin{equation*}
  \widetilde \rho:I^\bullet(G,V) = C^\bullet(G,V) \tensor_k V_{\triv} 
  \to I^\bullet(H,V) = C^\bullet(H,V) \tensor_k V_{\triv}.
\end{equation*}

We will be interested in these chain maps in the case where $H = P$ is
a parabolic subgroup of a reductive group.
\begin{prop}
  Let $G$ be a reductive algebraic group over $k$, let $V$ be a
  representation of $G$, and let $P$ be a parabolic subgroup of $G$.
  Then $\rho$ induces an isomorphism $H^i(G,V) \xrightarrow{\sim}
  H^i(P,V)$ for all $i \ge 0$.
\end{prop}

\begin{proof}
  Using Kempf's vanishing theorem, it is proved in
  \cite{JRAG}*{II.4.7} that 
  \begin{equation*}
    H^i(G,W) \simeq H^i(P,W) \quad \text{for any
      representation $W$ of $G$.}  
  \end{equation*}
  Note that $H^0(G,W) = W^G \subset W^P = H^0(P,W)$, so that if $W$ is
  finite dimensional it is immediate for dimension reasons that $W^G =
  W^P$. Since any $G$-representation is locally finite
  \cite{JRAG}*{I.2.13}, the equality $W^G = W^P$ in fact holds for
  \emph{all} $G$-representations $W$.

  For $m \ge 0$ and $i>0$, note that $H^i(G,I^m(G,V)) = 0$ since
  $I^m(G,V)$ is injective for $G$. But then $H^i(P,I^m(G,V)) = 0$ by
  the result just cited, so the resolution $I^\bullet(G,V)$ of $V$ is
  acyclic for the functor $(-)^P$. It follows that $H^\bullet(P,V)$
  may be computed as the cohomology of the complex
  $(I^\bullet(G,V))^P$.

  Now, the chain map $\widetilde \rho:I^\bullet(G,V) \to
  I^\bullet(P,V)$ is a quasi-isomorphism; thus
  \cite{weibel}*{Corollary 5.7.7, 5.7.9} shows that $\widetilde \rho$
  induces an isomorphism on hypercohomology
  \begin{equation*}
    \mathbb{H}^\bullet(P,I^\bullet(G,V))
    \to \mathbb{H}^\bullet(P,I^\bullet(P,V)).  
  \end{equation*}
  Since $I^\bullet(G,V)$ and $I^\bullet(P,V)$ are $(-)^P$-acyclic, the
  result just cited also shows that the hypercohomology groups
  coincide respectively with the cohomology of the complexes
  $I^\bullet(G,V)^P$ and $I^\bullet(P,V)^P$; this implies that
  \begin{equation*}
    \rho:C^\bullet(G,M) = (I^\bullet(G,V))^G = (I^\bullet(G,V))^P \to 
    C^\bullet(P,M) = (I^\bullet(P,V)^P
  \end{equation*}
  is a quasi-isomorphism, as required.
\end{proof}

\begin{stmt}
  \label{Z1-surjective}
  Let $P$ be a parabolic subgroup of $G$ and let $V$ be a linear
  representation of $G$.  Then the restriction mapping $(h \mapsto
  h_{\mid P}):Z^1(G,V) \to Z^1(P,V)$ is surjective, where $Z^1(G,V)
  \subset C^1(G,V)$ is the group of \emph{cocycles.}
\end{stmt}

\begin{proof}
  Indeed, if $f \in Z^1(P,V)$, the Proposition shows that there is
  some $h \in Z^1(G,V)$ for which $[f] = [h_{\mid P}]$ in $H^1(P,V)$.
  Thus there is some $v \in V = C^0(P,V)$ for which
  \begin{equation*}
    f = h_{\mid P} + \partial_P(v) 
  \end{equation*}
  where $\partial_P(v):P \to V$ given by $x \mapsto xv - v$.
  Now form the 1-cocycle $h_1 = h + \partial_G(v) \in Z^1(G,V)$
  where $\partial_G(v):G \to V$ is given by $g \mapsto gv-v$.
Then evidently $h_{1\mid P} = f$.
\end{proof}

\subsection{Group extensions}
\label{group-extensions}
When $V$ is a finite dimensional vector space (say, underlying a
module for an algebraic group), we will sometimes view $V$ as a
\emph{vector group} as in \S \ref{complement-to-R}; in other words, we
view $V$ as the linear algebraic group over $k$ whose functor of
points is given by the rule $V(\Lambda) = V \tensor_k \Lambda$ for
each commutative $k$-algebra $\Lambda$

Let $G$ be a linear algebraic group over $k$, and let $V$ be a finite
dimensional $G$-module. By an \emph{extension} of $G$ by $V$, we mean
a \emph{strictly exact} sequence of algebraic groups (see
\S\ref{terminology} for the terminology)
\begin{equation*}
  (*) \quad 0 \to V \xrightarrow{i} H \xrightarrow{\pi} G \to 1.
\end{equation*}  
We occasionally abuse terminology and refer to the extension $(*)$
simply as $H$, but the reader should be aware that $i$ and $\pi$ are
``part of the data'' of an extension of $G$ by $V$.

\begin{stmt}
  \label{split-condition}
  Consider an extension $H$ of $G$ by $V$ given by a sequence $(*)$.
  Let $C$ be a closed subgroup of $H$.  The following conditions are
  equivalent:
  \begin{enumerate}[(a)]
  \item Multiplication is an isomorphism $\mu:C \ltimes V \to H$.
  \item $\pi_{\mid C}:C \to G$ is an isomorphism of algebraic groups.
  \end{enumerate}
\end{stmt}

\begin{proof}
  Given (a), the map $\pi \circ \mu:C \ltimes V \to G$ is a separable
  surjection with kernel $\mu^{-1}(V) = V$, so that $\pi_{\mid C} =
  (\pi \circ \mu)_{\mid C}$ is an isomorphism; i.e. (b) holds.

  Given (b), set $\sigma = i \circ (\pi_{\mid C})^{-1}$, where $i$ is
  the inclusion $i:C \to H$. Then the map $H \to C \ltimes V$ given by
  $h \mapsto (\sigma \circ \pi(h),(\sigma \circ \pi(h))^{-1}h)$
  inverts $\mu$ so that (a) holds.
\end{proof}

A subgroup $C$ as in \ref{split-condition} is a \emph{complement to
  $V$ in $H$}.  We say that an extension $(*)$ of $G$ by $V$ is
\emph{split} if there is a complement to $V$ in $H$.

\subsection{Cohomology and group extensions}
\label{cohomology-and-group-extensions}
Let $V$ be a representation of the linear algebraic group $G$. We relate in this
section the cohomology group $H^2(G,V)$ and the extensions of $G$ by
the vector group $V$ in the sense of \S \ref{group-extensions}.

Consider an extension $H$ of $G$ by $V$ as in $(*)$ of \S
\ref{group-extensions}. Since the vector group $V$ is $k$-split
unipotent, Proposition \ref{complement-to-R} shows that there is a
regular function $s:G \to H$ which is a section to $\pi$.  Consider
now the regular mapping $\alpha_H:G\times G \to V$ defined by
\begin{equation*}
  s(g)s(g') = i(\alpha_H(g,g'))s(gg')
\end{equation*}
for all $g,g' \in G(\Lambda)$ and all commutative $k$-algebras $\Lambda$.

On the other hand, given a cohomology class $[\alpha] \in H^2(G,V)$
represented by a 2-cocycle $\alpha:G \times G \to V$, form 
\begin{equation*}
  E_\alpha = G \times V
\end{equation*}
with the operation $E_\alpha \times E_\alpha \to E_\alpha$ given by
\begin{equation*}
  (g,v)\cdot (g',v') = (gg', v + v' + \alpha(g,g')).
\end{equation*}
Let $i:V \to E_\alpha$ be the inclusion of the left-hand factor, and let
$\pi:E_\alpha \to G$ be the projection on the right-hand factor.

\begin{prop}
  \begin{enumerate}[(a)]
  \item For any extension $H$ of $G$ by $V$, we have $\alpha_H \in
    Z^2(G,V) \subset C^2(G,V)$, and the class $[\alpha_H]$ of
    $\alpha_H$ in $H^2(G,V)$ is independent of choices.
  \item The extension $H$ of $G$ by $V$ is \emph{split} if and only if
    $[\alpha_H] = 0$ in $H^2(G,V)$.
  \item For a 2-cocycle $\alpha \in Z^2(G,V)$, the indicated operation
    makes the variety $E=E_\alpha$ into a linear algebraic group which
    forms an extension of $G$ by $V$.  Moreover, $[\alpha_E] =
    [\alpha]$ in $H^2(G,V)$.
  \end{enumerate}
\end{prop}

\begin{proof}
   \cite{demazure-gabriel}*{II\S3.2.3} 
\end{proof}

\begin{rem}
  \begin{enumerate}[(a)]
  \item Let $G$ be a split, semisimple group over a perfect field $k$.
    Recall that by \ref{stmt:Witt-non-split}, there is a non-trivial
    extension of a $G$ by the Frobenius twist $\glie^{[1]}$ of its Lie
    algebra, viewed as a vector group with $G$-action.  Thus, the
    Proposition shows that $H^2(G_{/k},\glie^{[1]}) \ne 0$ for any
    split semisimple algebraic group $G$ over $k$.
  \item  Using the non-vanishing of
    a certain second cohomology group, we give in \S
    \ref{sub:levi-failure-cohomological} an example  of a
    linear algebraic group with no Levi decomposition
  \end{enumerate}
\end{rem}

\subsection{Conjugacy of complements}
Let $V$ be a representation for the linear algebraic group $G$, and form the
semidirect product $H = G \ltimes V$ which we view as the split
extension of $G$ by $V$. Write $\pi:H \to G$ for the quotient
homomorphism $(g,v) \mapsto g$, write $\sigma_0:G \to H$ for the
inclusion $g \mapsto (g,0)$ of $G$ in $H$, write $i:V \to H$ for the
inclusion $v \mapsto (1,v)$ of $V$ in $H$, and write $\psi:H = G
\ltimes V \to V$ for the regular map of varieties $(g,v) \mapsto v$;
i.e. $\psi$ is given by projection on the $V$ factor.

Let   
$\Sigma = \{ \sigma:G \to H \mid \text{$\sigma$ is both a section to $\pi$ and 
       a homomorphism of algebraic groups}\}.$
\begin{stmt}
  \label{param-sections}
  The assignment $\sigma \mapsto \psi \circ \sigma$ determines a
  bijection $\Sigma \to Z^1(G,V)$ where $Z^1(G,V)$ is the set of
  co-cycles in $C^1(G,V)$.
\end{stmt}

\begin{proof}
  Given $\sigma \in \Sigma$, one easily checks that $f = \psi \circ \sigma
  \in Z^1(G,V)$. Conversely, given $f \in Z^1(G,V)$ we may
  form the section $\sigma_f:G \to H = G \ltimes V$ given by the rule
  \begin{equation*}
    g \mapsto (g,f(g));
  \end{equation*}
  the co-cycle condition implies that $\sigma$ is a homomorphism of
  algebraic groups.

  It is clear that the assignments $\sigma \mapsto \psi \circ \sigma$
  and $f \mapsto \sigma_f$ are inverse to one another.
\end{proof}

\begin{stmt}
  \label{conj-of-complements}
  Let $\sigma \in \Sigma$, and let $f = \psi \circ \sigma \in
  Z^1(G,V)$ be the corresponding co-cycle as in \ref{param-sections}.
  The following are equivalent:
  \begin{enumerate}[(a)]
  \item $[f]=0$ in $H^1(G,V)$.
  \item The sections $\sigma,\sigma_0 \in \Sigma$ are conjugate by an
    element $i(v) \in H(k)$ for some $v \in V = V(k)$
  \item The sections $\sigma,\sigma_0 \in \Sigma$ are conjugate by an
    element of $H(\kalg)$ where $\kalg$ is an algebraic closure of
    $k$.
  \end{enumerate}
\end{stmt}

\begin{proof}
  By definition, the condition $[f]=0$ in $H^1(G,V)$ is equivalent to
  the condition $f = \partial v$ for some $v \in V = C^0(G,V)$ which
  simply means that $f:G \to V$ is given by the rule $f(g) = gv - v$.
  Now, for each commutative $k$-algebra $\Lambda$ and each $g \in
  G(\Lambda)$ we have
  \begin{equation*}
    i(v)\cdot \sigma_0(g) \cdot i(-v) = (1,v)(g,0)(1,-v) 
    = (g,gv - v) = (g,f(g)) = \sigma(g);
  \end{equation*}
  this proves the equivalence of (a) and (b).

  Evidently (b) implies (c). If (c) holds, say $\sigma =
  \Int(h)\sigma_0$ for $h \in H(\kalg)$, we may write $h = (g,v)$ for
  $v \in V(\kalg)$ and $g \in G(\kalg)$. Since $1_G = \pi \circ \sigma
  = \pi \circ \sigma_0$, evidently $g$ is \emph{central} so that
  in fact $\sigma = \Int(i(v))\sigma_0$.  But then $[f] = 0$ in
  $H^1(G_{/\kalg},V_{/\kalg})$ by the equivalence of (b) and (a) for
  $G_{/\kalg}$, whence (a) by application of
  \ref{cohomology-base-change}.
\end{proof}

\section{Some cohomological criteria for existence and conjugacy of Levi factors}
\label{Levi-criteria}
Throughout this section, $G$ will denote a \emph{connected} linear
algebraic group.  \emph{We suppose throughout \S \ref{Levi-criteria}
  that condition \textbf{(L)} of \S \ref{linearizable} holds for $G$.}

We fix a linearizable splitting sequence $1 = R_n \subset R_{n-1}
\subset \cdots \subset R_0 = R$ for $R$ and write $V_i$ for the linear
$\Gred$-representations for which there are $\Gred$-equivariant
isomorphisms $V_i \simeq R_i/R_{i-1}$.

\subsection{A cohomological criteria for conjugacy of Levi factors}
\label{levi-conjugacy}

\begin{theorem}
  Suppose that $G$ has a Levi factor, and suppose that
    \begin{equation*}
      H^1(\Gred,V_i) = 0 \quad \text{for} \quad 0 \le i \le n-1.  
    \end{equation*}
    Then any two Levi factors $L,L'$ of $G$ are
    conjugate by an element of $R(k)$.
\end{theorem}

\begin{proof}
  Proceed by induction on the length $n$ of a linearizable splitting
  sequence for $R$.  In case $n=1$, the result follows from
  \ref{conj-of-complements}.

  Now let $n>1$ and suppose the result holds for groups having
  linearizable splitting sequences of length $<n$.  Let $L,L' \subset G$ be Levi
  factors.

  The images $\overline{L},\overline{L'}$ of the Levi factors of
  $G$ in $G/R_1$  are Levi factor of $G/R_1$. Identifying
  $V_0$ and $R/R_1$, we view $G/R_1$ as an extension
  \begin{equation*}
    0 \to V_0 \to G/R_1 \to \Gred \to 1
  \end{equation*}
  and it follows from \ref{conj-of-complements} (i.e. the case $n=1$)
  that there is $v \in V_0$ such that $\overline{L} =
  \Int(i(v))\overline{L'}$.  Now, the morphism $R(k) \to
  (R/R_1)(k) = V_0$ is surjective by Proposition
  \ref{complement-to-R}. Choosing an $x \in R(k)$ which maps to $v \in V_0$
  and replacing $L'$ by $\Int(x)L'$ we may suppose that $\overline{L}
  = \overline{L'}$.

  Let $\pi:G \to G/R_1$ be the quotient mapping, and let $M =
  \pi^{-1}(\pi(L)) = \pi^{-1}(\overline{L})$. Then $M$ is an extension
  \begin{equation*}
    1 \to R_1 \to M \to \Gred \to 1;
  \end{equation*}
  since $\overline{L} = \overline{L'}$, we have $L,L' \subset M$ so
  that both $L$ and $L'$ are Levi factors of $M$. Since a splitting
  sequence for the unipotent radical $R_1$ of $M$ has length
  $n-1$, we may apply induction to learn that $L$ and $L'$ are
  conjugate by an element of $R_1(k)$.
\end{proof}

\subsection{A cohomological criteria for the existence of a Levi
  factor}
\label{easy-cohomology-existence}

\begin{theorem}
  Suppose that
    \begin{equation*}
      H^2(\Gred,V_i) = 0 \quad \text{for} \quad 0 \le i \le n-1.  
    \end{equation*}
    Then $G$ has a Levi factor.
\end{theorem}

\begin{proof}[Sketch:]
  One proves the result by induction on the length $n$ of a
  linearizable splitting sequence for $R$; the result in case $n=1$
  follows from Proposition \ref{cohomology-and-group-extensions}.  The
  induction step is similar to -- in fact, easier than -- that given
  in the proof of Theorem \ref{levi-conjugacy}; details are left to
  the reader.
\end{proof}

\subsection{Dimensional criteria for existence and conjugacy of Levi factors}
\label{small-modules}
Let $H$ be a reductive algebraic group over $k$.  A result of Jantzen
permits some control on the cohomology of $H$-modules whose dimension
is small relative to the size of the characteristic of $k$, as follows:
\begin{stmt}[Jantzen \cite{jantzen-semisimple}]
  \label{jantzen-CR}
  If $V$ is an $H$-module with $\dim V \le p$, then $V$ is completely reducible.
\end{stmt}

\begin{stmt}
  \label{jantzen-H1}
  If $V$ is a $H$-module with $\dim V < p$, then $H^1(H,V) = 0$.
\end{stmt}

\begin{proof}
  Using \ref{cohomology-base-change} we may as well suppose that $k$
  is algebraic closed. Suppose that 
  \begin{equation*}
    0 \ne H^1(H,V)  = \Ext^1_H(k,V).
  \end{equation*}
  Then a non-zero
  cohomology class determines a non-split extension of $G$-modules
  \begin{equation*}
    0 \to V \to E \to k \to 0;
  \end{equation*}
  in particular, the $H$-module $E$ is not completely reducible.  Now
  \ref{jantzen-CR} implies that $\dim E > p$ so that $\dim V \ge p$.
\end{proof}

\begin{stmt}
  \label{jantzen-higher-coh}
  Let $V$ be an $H$-module with $\dim V \le p$ and suppose that there
  are dominant weights $\mu_1,\dots,\mu_n$ such that the character of
  $V$ satisfies $\ch V = \sum_j \chi(\mu_j)$.
  Then $H^i(H,V) = 0$ for $i \ge 1$.
\end{stmt}

\begin{proof}
  The hypotheses imply that $\dim H^0(\mu_j) \le p$ for all $j$. Now
  Janzten's result \ref{jantzen-CR} implies that $V$ together with
  each $H^0(\mu_j)$ is completely reducible. In particular,
  $H^0(\mu_j) = L(\mu_j)$ is simple. Thus \ref{character-bases}
  implies that the composition factors of $V$ are precisely the
  $L(\mu_j)$. In turn, the complete reducibility of $V$ shows that $V$
  is the direct sum $V = \bigoplus_j L(\mu_j) = \bigoplus_j
  H^0(\mu_j)$, and the result now follows from \ref{standard-Weyl-ext}.
\end{proof}

\begin{theorem}
  Suppose that condition \textbf{(L)} holds for $G$ and write $p$ for
  the characteristic of $k$. Let $1 = R_n \subset R_{n-1} \subset \cdots
  \subset R_0 = R$ be a linearizable splitting sequence for $R$, and
  suppose that $R_i/R_{i-1} \simeq V_i$ for $G$-modules $V_i$.
  \begin{enumerate}[(a)]
  \item Suppose that $\dim R < p$. If $G$ has a Levi factor, then any
    two Levi factors are conjugate by an element of $R(k)$.
  \item Suppose that $\dim R \le p$. If there are dominant weights
    $\mu_1,\dots,\mu_m$ for $\Gred$ such that 
    \begin{equation*}
      \sum_{j=0}^{n-1} \ch V_j = \sum_{i=1}^m
      \chi(\mu_i) 
    \end{equation*}
    then $G$ has a Levi factor.
  \end{enumerate}
\end{theorem}

\begin{proof}
  Assertion (a) is a consequence of \ref{jantzen-H1} and Theorem
  \ref{levi-conjugacy}, and assertion (b) is a consequence of
  \ref{jantzen-higher-coh} and Theorem
  \ref{easy-cohomology-existence}.
\end{proof}

\begin{rem}
  Suppose that the reductive group $H$ has \emph{absolutely
    quasi-simple} derived group and semisimple rank $r$.  The main
  result of \cite{mcninch-dim-crit} shows that any $H$-representation
  with $\dim V \le rp$ is completely reducible; this leads to obvious
  analogues of \ref{jantzen-H1} and \ref{jantzen-higher-coh}. If $H$
  is isomorphic to the reductive quotient of $G$, the conclusion of
  the preceding Theorem remains valid after replacing in (a) the
  condition ``$\dim R < p$'' by the condition ``$\dim R < rp$'' and in
  (b) the condition ``$\dim R \le p$'' by the condition ``$\dim R \le
  rp$''.
\end{rem}
\subsection{Another criteria for the existence of a Levi factor}

\emph{Suppose in this section that $k$ is separably closed.}  Then the
reductive quotient $\Gred = G/R$ is \emph{split} and in particular has
a Borel subgroup over $k$.

Fix a Borel subgroup $B \subset \Gred = G/R$, and write $\widetilde B
= \pi^{-1}(B) \subset G$ where $\pi:G \to \Gred$ is the quotient
mapping. We may view $\widetilde B$ as an extension
\begin{equation}
\label{B-extension}
  1 \to R \to \widetilde{B} \to B \to 1.
\end{equation}
We assume that \eqref{B-extension} is split, and we fix a homomorphism
$\sigma:B \to \widetilde{B}$ which is a section to this extension. Write
$B_1$ for the image of $\sigma$. Under these assumptions, we find:
\begin{stmt}
\label{initial-criteria}
 $G$ has a Levi factor containing $B_1$.
\end{stmt}

\begin{proof}
  Recall that we have fixed a \emph{linearizable splitting sequence}
  $1 = R_n \subset R_{n-1} \subset \cdots \subset R_0 = R$ for $R$;
  for each $i$, the vector group $R_i/R_{i+1}$ is
  $\Gred$-equivariantly isomorphic to the linear representation $V_i$ of
  $\Gred$.

  We prove the result by induction on $n$. First suppose that $n=1$;
  then $R \simeq V_0 = V$ is a linearizable vector group. Viewing $G$
  as an extension of $\Gred$ by $V$, we find a 2-cocycle $\alpha =
  \alpha_G \in Z^2(G,V)$ as in \S
  \ref{cohomology-and-group-extensions}. The restriction to $B$ (more
  precisely: to $B \times B$) of $\alpha$ is the 2-cocycle determined
  by the extension \eqref{B-extension}; thus by hypothesis and by
  Proposition \ref{cohomology-and-group-extensions}, we have
  $[\alpha_{\mid B\times B}] = 0$ in $H^2(B,V)$. But then $[\alpha] =
  0$ in $H^2(G,V)$ by Proposition \ref{parabolic}, so another
  application of Proposition \ref{cohomology-and-group-extensions}
  shows that the extension
  \begin{equation*}
    0 \to V \to G \to \Gred \to 1
  \end{equation*}
  is split. Let $\tau:\Gred \to G$ be a homomorphism which is a
  section to this extension. Restricting $\tau$ to the Borel group $B$
  gives a section $\tau_{\mid B}:B \to \widetilde{B}$. According to
  \ref{param-sections}, there is $f \in Z^1(B,V)$ for which
  \begin{equation*}
    \sigma(b) = \tau(b) \cdot i(f(b)) \quad 
    \text{for each commutative $k$-algebra $\Lambda$ 
      and each $b \in B(\Lambda)$.}
  \end{equation*}
  According to \ref{Z1-surjective} we may find $h \in Z^1(G,V)$ such
  that $f = h_{\mid B}$.  Now the image of the section $\Gred \to G$
  given by $g \mapsto \tau(g) \cdot i(h(g))$ is a Levi factor of $G$
  containing $B_1$, as required.
  
  Now suppose $n>1$ and that the conclusion of the Theorem is known
  for groups having linearizable splitting sequences of length $<n$.
  Let $\overline{B_1}$ be the image of $B_1$ in $G/R_1$. Then
  $\overline{B_1}$ is a complement to $R/R_1$ in $\widetilde B/R_1$,
  so that the extension
  \begin{equation}
    \label{factor-B-extension}
    0 \to R/R_1 \to \widetilde B/R_1 \to B \to 1 
  \end{equation}
  is split; thus by the case $n=1$, we may find a Levi factor
  $\overline{M} \subset G/R_1$ containing $\overline{B_1}$.

  Write $\pi:G \to G/R_1$ for the quotient mapping, and consider
  \begin{equation*}
    \widetilde{B_1} = \pi^{-1}(\overline{B_1}) \subset M = \pi^{-1}(\overline{M}).
  \end{equation*}
  After identifying $\overline{M}$ with $\Gred$, we may view $M$ and $\widetilde{B_1}$ as
  extensions
  \begin{equation*}
    1 \to R_1 \to M \to \Gred \to 1 \quad \text{and} \quad
    1 \to R_1 \to \widetilde{B_1} \to B \to 1.
  \end{equation*}
  Moreover, $B_1 \subset \widetilde{B_1}$ so that $\sigma:B \to
  \widetilde{B_1}$ is a section to the second extension.  Since the
  unipotent radical $R_1$ of the linear algebraic group $M$ has a
  linearizable splitting sequence of length $n-1$, induction gives a
  Levi factor $N$ of $M$ containing $B_1$, and evidently $N$ is the
  desired Levi factor of $G$.
\end{proof}

\subsection{Existence of a unique Levi factor containing a fixed maximal torus}
\label{main-existence-result}
We prove in this section the existence result for Levi factors that is
applied below in \S \ref{parahoric-section}.

Fix a maximal torus $T$ of $G$, and let $\Tred$ denote the image of
$T$ under $\pi$; then $\Tred$ is a maximal torus of $\Gred$.

\begin{stmt}
  \label{stmt:sep-closed-enough}
  Suppose that there is a unique Levi factor $M$ of $G_{/\ksep}$
  containing $T_{/\ksep}$. Then $M$ is defined over $k$ and is the
  unique Levi factor of $G$ containing $T$.
\end{stmt}

\begin{proof}
  Write $\Gamma$ for the Galois group of $\ksep/k$. Since $T$ is
  defined over $k$, it is stable by the action of 
  $\Gamma$. The unicity of $M$ then shows that $M$ is $\Gamma$ stable;
  thus $M$ is defined over $k$ by Galois descent
  \cite{springer-LAG}*{11.2.8}.
\end{proof}

\begin{theorem}
  Let $k$ be separably closed, fix a Borel subgroup $B \subset \Gred$
  containing $\Tred$, and write $\widetilde B = \pi^{-1}(B)$ where
  $\pi:G \to \Gred = G/R$ is the quotient mapping.  Assume that there
  is a unique closed subgroup $B_1 \subset \tilde B$ such that $T
  \subset B_1$ and such that $\pi_{\mid B_1}:B_1 \xrightarrow{\sim} B$
  is an isomorphism. Then $G$ has a unique Levi factor containing $T$.
\end{theorem}

\begin{proof}
  Since $k$ is assumed to be separably closed and since the extension
  $1 \to R \to \widetilde{B} \to B \to 1$ is split by hypothesis, we
  may apply \ref{initial-criteria}; it gives a Levi factor $M \subset
  G$ containing $B_1$ and in particular containing $T$. It only
  remains to argue the uniqueness.

  Before proceeding further, we observe first that if $B'$ is
  \emph{any} Borel subgroup of $\Gred$ (over $k$) containing $\Tred$,
  the hypothesis shows: $(*)$ there is a unique closed $B'_1 \subset
  \pi^{-1}(B')$ such that $\pi_{\mid B'_1}:B_1' \xrightarrow{\sim} B$
  is an isomorphism. Indeed, since $k$ is separably closed,
  \cite{borel-LAG}*{Theorem V.20.9} shows that any two Borel subgroups
  of $\Gred$ are conjugate by an element of $\Gred(k)$, and the claim
  follows since the mapping $G(k) \to \Gred(k)$ is surjective by
  Proposition \ref{complement-to-R}.

  Suppose now that $M,M'$ are two Levi factors of $G$ each containing $T$.
  The assumption shows that both $M$ and $M'$ contain $B_1$. Choose in
  the reductive group $M$, resp. $M'$, a Borel subgroup $C$, resp.
  $C'$, containing $T$ opposite to $B_1$

  Then $\pi(C)$ and $\pi(C')$ are Borel subgroups of $\Gred$
  containing $T$ and opposite to $B$. Hence $\pi(C)$ and $\pi(C')$ are
  \emph{equal} by \cite{borel-LAG}*{Prop. IV.14.21(i)}. But then $C =
  C'$ by the observation $(*)$.  Then $M$ and $M'$ each contain the
  variety $CB_1 = C'B_1$ as an open, dense subset
  \cite{borel-LAG}*{Prop.  1V.14.21(iii)}, hence $M = M'$ as required.
\end{proof}

\section{Parahoric group schemes}
\label{parahoric-section}
Let $\A$ be a Henselian discrete valuation ring (DVR) with maximal
ideal $\mm = \varpi \A$. Henselian means that $\A$ satisfies the
conclusion of Hensel's Lemma; for example, the DVR $\A$ is Henselian
if it is complete in its $\mm$-adic topology.  We write $K$ for the
field of fractions of $\A$ and $k$ for the residue field of $\A$;
\emph{we assume throughout \S\ref{parahoric-section} that the residue
  field $k$ is perfect}.

We denote by $G$ a connected and reductive group over the field $K$.
We will be interested in certain algebraic groups over $k$ which arise
in the study of this group $G$. These are the special fibers
$\parah_{/k}$ of certain smooth $\A$-group schemes $\parah$
associated with $G$ known as \emph{parahoric group schemes}.

As hinted at in the introduction, these parahoric group schemes are
related to the stabilizers in $G(K)$ of certain subsets of the
\emph{affine building} $\Build$ of $G$. The reader is referred to
\cite{corvallis}, \cite{bruhat-tits-I} and \cite{bruhat-tits-II} for
more on these matters.

For our purposes in this paper, we do not need to discuss in detail
this affine building. When $G$ is split over $K$, we will just 
describe the parahoric group schemes associated to the ``standard
alcove'' in a fixed apartment; up to isomorphism of $\A$-group
schemes, these account for all parahoric group schemes associated with
$G$.  When $G$ splits over an unramified extension $L/K$, every
parahoric group scheme for $G$ is obtained by ``\'etale descent'' from
one for the split group $G_{/L}$.

These descriptions will permit us to give a proof of Theorem A from \S
\ref{introduction}.

\subsection{An apartment in case $G$ is split over $K$}
\label{apartment-in-split-case}

Suppose in \S\ref{apartment-in-split-case} that $G$ is split over $K$
and that $T$ is a maximal split $K$-torus of $G$.

In order to describe the parahoric group schemes associated with $G$,
we must first describe a part of the Bruhat-Tits building $\Build$
associated with $G$, namely, the \emph{apartment} $A$ corresponding
to $T$.

Write $\Phi \subset X^*(T)$ for the roots of $G$, and for $a \in \Phi$
write $U_a$ for the corresponding root subgroup of $G$.  Fix a
\emph{Chevalley system}, i.e. a choice of $K$-isomorphisms $\chi_a:\Ga
\to U_a$ for each $a \in \Phi$ satisfying the ``usual'' commutation
relations; see \cite{bruhat-tits-II}*{(3.2.2)}.

Now write $A = X_*(T) \tensor \R$, where $X_*(T)$ is the lattice of
\emph{co-characters} of the split torus $T$. An \emph{affine
  function} on $A$ has the form $\phi = (\psi,r)$ where $\psi \in
A^\vee = X^*(T) \tensor \R$ and $r\in \R$; thus $\phi(a) = \psi(a) +
r$ for $a \in A$, and $\psi$ is the \emph{gradient} of $\phi$. The
affine roots $\Phi^{\aff}$ of $G$ with respect to $T$ are the affine
functions on $A$ of the form $\alpha = (a,\gamma)$ for $a \in \Phi$
and $\gamma \in \Z$.

The \emph{walls} of $A$ are the subsets $\alpha^{-1}(0)$ for $\alpha
\in \Phi^{\aff}$; the connected components of the complement of the
union of all walls are the \emph{alcoves} of $A$, and the facets
of these alcoves are the facets of $A$.

Write $\Phi = \bigcup_{i \in I} \Phi_i$ as the disjoint union of its
irreducible components $\Phi_i$, fix for each $i$ a \emph{basis of
  simple roots} $\Delta_i$ for the root system $\Phi_i$, and let
$\tilde \alpha_i \in \Phi_i$ be the highest root (i.e. the dominant
long root for the given choice of basis).  Now write $\Delta_i^0$ for
the set of affine roots
\begin{equation*}
  \Delta_i^0 = \{(a,0) \mid a \in \Delta_i\} \cup \{(-\tilde \alpha_i,1)\}
\end{equation*}
and write $\Delta^0 = \bigcup_{i \in I} \Delta_i^0$.

The \emph{fundamental alcove} is the set $\displaystyle C =
\bigcap_{\alpha \in \Delta^0} \alpha^{-1}((0,\infty))$; when $G$ is
(quasi-) simple, $C$ is a simplex; for semisimple $G$, $C$ is a direct
product of simplices, and in general $C$ is a direct product of
simplices and a Euclidean space.

\begin{stmt}
  \label{facet-description}
  The facets in the closure of the fundamental alcove $C$ are in
  bijection with the subsets $\Theta = \cup_{i \in I} \Theta_i \subset
  \Delta^0$ where $\Theta_i$ is a non-empty subset of $\Delta_i^0$ for
  each $i \in I$.  The facet $\Facet_\Theta$ determined by the subset
  $\Theta$ is given by
  \begin{equation*}
    \Facet_\Theta = \left ( \bigcap_{\alpha \in \Theta} 
      \alpha^{-1}((0,\infty)) \right)\cap 
    \left ( \bigcap_{\alpha \in \Delta^0 \setminus  \Theta} 
      \alpha^{-1}(0) \right).
  \end{equation*}
\end{stmt}

\begin{proof}
  The assertion is easily reduced to the case where $\Phi$ is
  irreducible, and in that case it follows from
  \cite{bourbaki}*{Corollary VI.2.3}.
\end{proof}

\begin{stmt}
  \label{closed-system-of-roots}
  Let $\Phi_\Theta$ be the set of all roots $a \in \Phi$ such that
  there is $\alpha \in \Phi^{\aff}$ with gradient $a$ for which
  $\alpha_{\mid \Facet} = 0$.  Then $\Phi_\Theta$ is a closed system
  of roots in $\Phi$. The Dynkin diagram of $\Phi_\Theta$ is obtained
  from the extended Dynkin diagram of $\Phi$ by deleting the nodes
  corresponding to $\Theta$ and all edges adjacent to those nodes.
\end{stmt}

\begin{proof}
  If $a,b \in \Phi_\Theta$ then there are integers $\gamma_a,\gamma_b$
  such that the affine roots $(a,\gamma_a)$ and $(b,\gamma_b)$ vanish
  on $\Facet$. Now, if $a+b \in \Phi$ then evidently
  $(a+b,\gamma_a+\gamma_b)$ vanishes on $\Facet$ so that $a + b \in
  \Phi_\Theta$, as required. For the second assertion, see (e.g.)
  \cite{Prasad-Raghunathan}*{2.22}.
\end{proof}

\begin{rem}
  \begin{enumerate}[(a)]
  \item   Strictly speaking, the apartment $A$ of the affine building $\Build$
    of $G$ associated with the torus $T$, should be instead an
    \emph{affine space} under $X_*(T) \tensor \R$; our choice of
    Chevalley system determines an ``origin'' for $A$ and permits our
    identification of $A$ with $X_*(T) \tensor \R$; see \S1 of
    \cite{corvallis} for a nice account of these matters.
  \item If $W = N_G(T)/T$ denotes the Weyl group of $G$ and $W^{\aff}
    = W \ltimes X_*(T)$ the affine Weyl group, then any alcove in $A$
    is conjugate by $W^{\aff}$ to the fundamental alcove $C$; for this
    reason, we focus attention on the alcove $C$.
  \end{enumerate}
\end{rem}

\subsection{Parahoric group schemes in case $G$ is split over $K$}
\label{split-parahoric-description}

Keep the assumptions and notations of \ref{apartment-in-split-case},
and fix a facet $\Facet$ of the standard alcove $C$. Thus $\Facet =
\Facet_\Theta$ for some $\Theta$ as in \ref{facet-description}.  We
are going to describe the parahoric group scheme associated with the
facet $\Facet$.

This parahoric group scheme is constructed from $\A$-group schemes
attached to $T$ and the $U_a$.  Since $T$ is a split $K$-torus, there
is a canonical split $\A$-torus $\TT$ with generic fiber $\TT_{/K} =
T$ for which $X^*(T) = X^*(\TT)$; see \cite{bruhat-tits-II}*{1.2.11}.

Given $\alpha = (a,\gamma) \in \Phi^{\aff}$ with gradient $a \in
\Phi$, we get a \emph{subgroup} $U_\alpha \subset U_a(K)$ by the rule
$U_\alpha = \chi_a(\varpi^\gamma\A)$ where $\chi_a:\Ga \to U_a$ is the
isomorphism involved in the Chevalley system. We get moreover a smooth
$\A$-group scheme $\UU_\alpha$ with generic fiber $U_a$ obtained by
transport of structure via $\chi_a$ from the $\A$-group scheme defined
by the $\A$-module $\varpi^\gamma\A$ as in
\cite{bruhat-tits-II}*{(1.4.1)}; then $\UU_\alpha(\A) = U_\alpha$.

For $a \in \Phi^+$, let $e_a = 0$ if the root $a$
vanishes on $\Facet \subset A$, and let $e_a = 1$ otherwise; note
in the latter case that $1 > a(x) > 0$ for each $x \in \Facet$.
Consider the following set of affine roots:

\begin{equation*}
  \Psi=\Psi_\Theta = \{(a,0) \mid a \in \Phi^+\} \cup \{(-a,e_a) \mid a \in  \Phi^+\}
\end{equation*}

\begin{stmt}
  \label{parahoric-construction}
  $(\TT,(\UU_\alpha)_{\alpha \in \Psi_\Theta})$ is a \emph{schematic
    root datum} over $\A$. There is a smooth $\A$-group scheme
  $\parah$ with generic fiber $G$ such that the injections of $T$ and
  the $U_a$ in $G$ prolong to isomorphisms of $\TT$ and the $\UU_a$
  onto closed $\A$-subschemes of $\parah$. With the obvious
  definitions of the sets of affine roots $\Psi^{\pm}$, the product
  mappings yield isomorphisms of $\prod_{\alpha \in \Psi^\pm}
  \UU_{\alpha}$ onto closed subgroup schemes $\UU^\pm$ of $\parah$,
  and the product mapping yields an isomorphism of $\UU^- \times \TT
  \times \UU^+$ onto an open subscheme of $\parah$.
\end{stmt}

\begin{proof}
  This follows from \cite{bruhat-tits-II}*{3.8.1, 3.8.3, \S4.6 and
    4.6.26}; we only need to observe that the function $f_\Facet$ of
  \emph{loc. cit.}, 4.6.26 takes the value $0$ on any $a\in \Phi^+$,
  while for $a \in \Phi^+$, $f_\Facet(-a) = 0$ if $e_{a} = 0$ and
  $0 < f_\Facet(-a) < 1$ if $e_{a} = 1$.
\end{proof}

\begin{stmt}[\cite{Prasad-Raghunathan}*{2.22} or \cite{bruhat-tits-II}*{4.6.12}]
  \label{parahoric-reductive-quotient}
  The root system of the reductive quotient of the special fiber
  $\parah_{/k}$ is  $\Phi_\Theta$.
\end{stmt}

\begin{stmt}
  \label{has-unique-borel-lift}
  Fix a Borel subgroup $B$ containing $\TT_{/k}$ in the reductive
  quotient $\parah_{/k,\red}$ of the special fiber of $\parah$, and
  let $\tilde B = \pi^{-1}(B) \subset \parah_{/k}$ where $\pi$ is the
  quotient mapping.  There is a unique subgroup $B_1 \subset \tilde B$
  such that $\TT_{/k} \subset B_1$ and such that $\pi_{\mid B_1}:B_1
  \to B$ is an isomorphism.
\end{stmt}

\begin{proof}
  The existence of $B_1$ follows from \ref{closed-system-of-roots}
  together with \cite{bruhat-tits-II}*{Cor. 4.6.4(ii)}, since a
  positive system of roots in $\Phi_\Theta$ is quasi-closed in the
  terminology of \emph{loc. cit.}.  Uniqueness follows from \emph{op.
    cit.}  Cor. 4.6.4(i) since any $B_1$ must contain the root
  subgroup of $\parah_{/k}$ for each $a \in \Phi_\Theta^+$.
\end{proof}

\subsection{Condition \textbf{(L)} for the special fiber of a parahoric group 
  scheme in the  split case}
\label{split-parahoric-condition-(L)}

Keep the notation and assumptions of
\S\ref{split-parahoric-description}, and suppose in addition that the
residue field $k$ of $\A$ is \emph{algebraically closed}.

In this case, we verify that condition \textbf{(L)} holds for the
special fiber $\parah_{/k}$ of the parahoric group scheme
$\parah_\Facet$ determined by the facet $\Facet = \Facet_\Theta$.

Recall that $\Phi = \cup_{i \in I} \Phi_i$.
\begin{stmt}[see e.g. \cite{Prasad-Raghunathan}*{2.10}]
  \label{affine-expression}
  For an affine root $\alpha \in \Phi^{\aff}$, let $i \in I$ such that
  the gradient of $\alpha$ is in $\Phi_i$. Then we may write
  $\displaystyle \alpha = \sum_{\beta \in \Delta_i^0} t_\beta \beta$
  for unique integers $t_\beta$, and the $t_\beta$ are either all
  non-negative or all non-positive.
\end{stmt}

We define a function $\ell_\Theta$ on affine roots as follows: if
$\alpha \in \Phi^{\aff}$ has gradient $a \in \Phi_i \subset \Phi$, we
set
\begin{equation*}
  \ell_\Theta(\alpha) = \sum_{\alpha \in \Theta_i} t_\alpha
\end{equation*}
where the $t_\alpha$ are as in   \ref{affine-expression}.

Write $\delta = (0,1)$ for the constant affine function on $A$.
Then for $i \in I$, we may write
  $\delta =  \sum_{\alpha \in \Delta^0_i} m(i)_\alpha \alpha$
for unique positive integers $m(i)_\alpha$,
and we set 
\begin{equation}
  \label{ell-equation}
  \ell_i = \sum_{\alpha \in \Delta^0_i} m(i)_\alpha \ \text{for $i \in I$}.
\end{equation}

\begin{stmt}
  \label{parahoric-condtion-(L)}
  The special fiber $\parah_{/k}$ satisfies condition \textbf{(L)} of
  \S \ref{linearizable}.
\end{stmt}

\begin{proof}
  It follows from \cite{Prasad-Raghunathan}*{2.22} that the unipotent
  radical $R$ of $\parah_{/k}$ is generated by the canonical images in
  $\parah(k)$ of the groups $\UU_\alpha(\A)$ for all affine roots
  $\alpha$ with $\ell_\Theta(\alpha) \ge 1$.

  Moreover, it is straightforward to verify that $R = \prod_j R(j)$ is
  the direct product of the $G$-invariant subgroups $R(j)$ generated
  by the canonical images of those groups $\UU_\alpha(\A)$ as above
  for which $\alpha$ has gradient in $\Phi_j$.  Thus it suffices to
  prove the assertion in the case where $\Phi = \Phi_1$ is
  irreducible, which we now suppose. Write $\ell$ for $\ell_1$ as in
  \eqref{ell-equation}.

  Let $0 \le i \le \ell-1$. Since $k$ is algebraically closed, the
  canonical images in $\parah(k)$ of the groups $\UU_{\alpha}(\A)$
  for which $\ell_\Theta(\alpha) \ge i+1$ generate a closed, connected
  subgroup $R_i \subset \parah_{/k}$.

  It follows from \cite{Prasad-Raghunathan}*{2.16} that $R_{\ell-1} =
  1$ and \emph{loc. cit.}, 2.22 shows for $0 \le i \le \ell -2$ that
  the quotient $R_i/R_{i+1}$ has a natural structure of a finite
  dimensional $k$-vector space.  The usual commutator relations
  \cite{springer-LAG}*{8.2.3} for $G$ imply that the reductive
  quotient of $\parah_{/k}$ acts linearly on each quotient, thus
  indeed \textbf{(L)} holds for $\parah_{/k}$.
\end{proof}

\subsection{The maximal unramified extension}
Fix a separable closure $\Ksep$ of $K$, and let $\Kun$ denote the
maximal unramified extension of $K$ in $\Ksep$.  The integral closure
$\Aun$ of $\A$ in $\Kun$ is a DVR whose residue field is an algebraic
closure $\kalg$ of the perfect field $k$.

We consider the reductive $\Kun$-group $\Gun$ obtained from $G$ by
base-change.  By a theorem of Lang \cite{SerreGC}*{II.3.3}, $\Kun$ is
a $C_1$ field. Thus, we have the following important result due to
Steinberg (in case $K$ is perfect) and Borel-Springer:
\begin{stmt}[\cite{borel-springer}]
  \label{Gun-quasisplit}
  The reductive group $\Gun$ is \emph{quasi-split}; i.e.
  $\Gun$ has a Borel subgroup defined over $\Kun$.
\end{stmt}

When $\Gun$ is split over $\Kun$, we have the following result:
\begin{stmt}[\cite{bruhat-tits-II}*{5.1.12}]
  \label{splitting-torus}
  There is a maximal $K$-split torus $S$ of $G$ and a maximal torus
  $T$ of $G$ containing $S$ for which $T_{/\Kun}$ is $\Kun$-split.
\end{stmt}

\subsection{Parahoric group schemes for $G$ and \'etale descent}
Let $G$ be a connected, reductive group over $K$. 

If $G$ is \emph{quasi-split} - i.e. has a Borel subgroup over $K$ -
the parahoric group schemes associated with $G$ can be described in an
explicit manner analogous to our description in
\S\ref{split-parahoric-description} when $G$ is split; see
\cite{bruhat-tits-II}*{\S4}.

In general, the parahoric group schemes for $G$ arise by \emph{\'etale
  descent} from those of $\Gun$. We will describe this descent when
$\Gun$ is split over $\Kun$.  Fix $S \subset T \subset G$ as in
\ref{splitting-torus}. Then $T_{\Kun}$ is a $\Kun$-split torus of
$\Gun$; write $\Aptun$ for the apartment associated with $\Tun$ as in
\S\ref{apartment-in-split-case}. The Galois group $\Sigma =
\Gal(\Kun/K)$ acts on $\Aptun$, and we have the following:
\begin{stmt}[\'Etale descent]
  \label{etale-descent-result}
  Let $\Facet$ be a $\Sigma$-invariant facet of $\Aptun$, and let
  $\parahun = \parah_{\un,\Facet}$ be the corresponding parahoric
  group scheme for $\Gun$. There is a smooth $\A$-group scheme
  $\parah$ whose generic fiber $\parah_{/K}$ may be identified with
  $G$ for which $\parahun = \parah \tensor_\A \Aun$.
\end{stmt}

\begin{proof}
  It follows from \cite{bruhat-tits-II}*{4.6.30} that the action of
  $\Sigma$ on $\Kun[\Gun] = K[G] \tensor_K \Kun$ leaves invariant the
  subalgebra $\Aun[\parahun]$, the coordinate algebra of $\parahun$.
  Thus \cite{bruhat-tits-II}*{5.1.8} shows that $\parahun$ arises by
  base-change $\A \to \Aun$ from a smooth $\A$-group scheme $\parah$.
\end{proof}

The parahoric group schemes associated with $G$ are precisely the
$\A$-group schemes obtained via \ref{etale-descent-result}; see
\cite{bruhat-tits-II}*{\S5.2}.

Let $\parahun = \parah \tensor_\A \Aun$ be a parahoric group scheme
corresponding to a $\Sigma$-invariant facet in $\Aptun$.  The
special fiber ${\parahun}_{/\kalg}$ of $\parahun$ is a linear
algebraic group over $\kalg$, and the special fiber $\parah_{/k}$ of
$\parah$ is a linear algebraic group over $k$. We now have:

\begin{stmt}
  \label{special-fiber-base-change}
  The special fibers of $\parahun$ and $\parah$ are related by
  ${\parahun}_{/\kalg} = (\parah_{/k})_{/\kalg}.$
\end{stmt}

\subsection{Levi factors for the special fiber of $\parah$}
We prove in this section the Theorem stated in the introduction:

\begin{proof}[Proof of Theorem A]
  We first consider the case where $K = \Kun$ and $G$ is split over
  $K$.  Let $\parah$ be a parahoric group scheme associated with $G$.
  Up to isomorphism of $\A$-group schemes, $\parah$ arises via the
  construction in \S\ref{split-parahoric-description}. 

  Since $k = \kalg$, \ref{parahoric-condtion-(L)} shows that the
  special fiber $\parah_{/k}$ satisfies condition \textbf{(L)}.
  Moreover, by \ref{has-unique-borel-lift} $\parah_{/k}$ satisfies the
  hypotheses of Theorem \ref{main-existence-result}. It follows that
  $\parah_{/k}$ has a unique Levi factor containing a fixed maximal
  torus $T$. In particular, since $k=\kalg$ it follows that any two
  Levi factors of $\parah_{/k}$ are conjugate by an element of
  $\parah_{/k}(k)$.

  Returning to the general case, the parahoric group scheme $\parah$
  associated with $G$ arises by \'etale descent from a parahoric
  $\Aun$-group scheme $\parahun$ associated with $\Gun$
  \ref{etale-descent-result}; thus $\parahun = \parah \tensor_\A \Aun$
  and $\parah_{\un/\kalg} = \parah_{/k} \tensor_k \kalg =
  \parah_{/\kalg}.$

  Thus for any maximal $k$-torus $T$ of $\parah_{/k}$, the result
  above for $\parah_{\un/\kalg}$ together with
  \ref{stmt:sep-closed-enough} shows that $\parah_{/k}$ has a unique
  Levi factor (over $k$) containing $T$.

  If $G$ is split over $K$, then $\parah_{/k}$ has a maximal torus
  which is $k$-split, and any Levi factor of $\parah_{/k}$ contains a
  maximal split torus.  The existence assertions in parts (i) and (ii)
  of the Theorem now follow at once.  A result of Borel-Tits shows
  that any two maximal split tori of $\parah_{/k}$ are conjugate by an
  element of $\parah(k)$ \cite{conrad-gabber-prasad}*{Theorem C.2.3}
  or \cite{springer-LAG}*{15.2.6}; the conjugacy assertion in (i) now
  follows.  The geometric conjugacy assertion in (ii) follows from the
  observation that all maximal tori of $\parah_{/k}$ are geometrically
  conjugate.
\end{proof}

\section{Examples}

\subsection{Cohomological construction of a linear algebraic group with no Levi decomposition}
\label{sub:levi-failure-cohomological}

For a reductive group $G$, the non-vanishing of some $H^2(G,-)$ may be
used to construct linear algebraic groups having no Levi factor. We
give here an example.

Let $k$ be algebraically closed field of characteristic $p \ge 3$, and
let $G = \SL_{3/k}$. Write $T$ for the diagonal maximal torus in $G$,
and $B$ for the Borel subgroup of lower triangular matrices.

Let $\alpha_1,\alpha_2 \in X^*(T)$ be the simple roots of $G$ for the
choice of $B$ (the weights of $T$ on $\Lie(B)$ are non-positive), and
let $\varpi_1,\varpi_2 \in X^*(T)$ be the fundamental dominant weights
corresponding to the $\alpha_i$.

\begin{stmt}
  \label{low-alcove-simple}
  If $\tau$ is a dominant weight for which $\langle
  \tau,\alpha_1^\vee + \alpha_2^\vee \rangle \le p-2$, then $\rad
  V(\tau) = 0$ so that $L(\tau) = V(\tau) = H^0(\tau)$.
\end{stmt}
\begin{proof}
  The condition means that $\tau$ is in the closure of the lowest
  alcove for the dot-action of the affine Weyl group $W_p$ on
  $X^*(T)$; see \cite{JRAG}*{II.6.2(6)}. Since that alcove closure is
  a fundamental domain for the indicated action, one knows for any
  dominant weight $\gamma \le \tau$ that $\gamma \not \in W_p \bullet
  \tau$; thus by \ref{ext-by-Weyl-radical} and the linkage principle
  \cite{JRAG}*{II.6.17}, $\Hom_G(\rad V(\tau),L(\gamma)) \simeq
  \Ext^1_G(L(\tau),L(\gamma)) = 0$. Since $\tau$ is the highest weight
  of $V(\tau)$, it follows that $\rad V(\tau) = 0$ as required.
\end{proof}

According to \cite{JRAG}*{II.8.19}, the $G$-module $V(\tau)$ has a
filtration 
\begin{equation*}
  V(\tau) = V^0 \supset V^1 \supset V^2 \supset \cdots 
\end{equation*}
for which
$V^1 = \rad V(\tau)$ and $J(\tau) = \sum_{i \ge 1} \ch V^i$ is given by the Jantzen
sum formula \cite{JRAG}*{II.8.19}.
The following result follows from \ref{character-bases}:
\begin{stmt}
  \label{simple-Jantzen-sum} 
  Let $\tau$ be a dominant weight and suppose that $J(\tau) = \ch L$
  for a simple $G$-module $L$. Then $\rad V(\tau) \simeq L$.
\end{stmt}

We are going to apply the Jantzen sum formula to achieve some
cohomology computations.  Let
\begin{equation*}
  \lambda = p\varpi_1, \quad \mu = (p-2)\varpi_1 + \varpi_2, \quad \gamma = (p-3)\varpi_1.
\end{equation*}
\begin{stmt}
  \label{V(mu)}
  $\rad V(\mu) = L(\gamma)$, and
  $\ch L(\mu) = \chi(\mu) - \chi(\gamma).$
\end{stmt}
\begin{proof}
  Evaluating the Jantzen sum formula, we find that $J(\mu) =
  \chi(s_{\alpha_1+\alpha_2,p}\bullet \mu) = \chi(\gamma).$
  Since $\langle \gamma, \alpha_1^\vee + \alpha_2^\vee \rangle = p-3$,
  \ref{low-alcove-simple} shows that $V(\gamma) = L(\gamma)$ so that
  $\ch L(\gamma) = \chi(\gamma)$. The assertion now follows from
  \ref{simple-Jantzen-sum}.
\end{proof}

\begin{stmt}
  \label{V(lambda)}
  $\rad V(\lambda) =
  L(\mu)$, and $\ch L(\lambda) =
  \chi(\lambda) -\chi(\mu) + \chi(\gamma)$.
\end{stmt}
\begin{proof}
  Evaluating the Jantzen sum formula and using \ref{V(mu)}, we find that
  \begin{equation*}
    J(\mu) = \chi(s_{\alpha_1,p}\bullet \lambda) + \chi(s_{\alpha_1+\alpha_2,p}
    \bullet \lambda) = \chi(\mu) - \chi(\gamma)  = \ch L(\mu),
  \end{equation*}
  so the result follows from \ref{simple-Jantzen-sum}.
\end{proof}

\begin{stmt}
  \label{non-vanishing-ext2-example}
  $\dim \Ext^2_G(L(\lambda),L(\gamma)) = 1$.
\end{stmt}

\begin{proof}
  Since $L(\gamma) = H^0(\gamma)$, it follows from
  \ref{standard-Weyl-ext} that $\Ext^i_G(V(\lambda),L(\gamma)) = 0$
  for $i \ge 1$.  Using \ref{V(mu)} and \ref{V(lambda)} together with
  \ref{ext-by-Weyl-radical}, we find that $\dim
  \Ext^1_G(L(\mu),L(\gamma)) = 1$ and $\Ext^1_G(V(\lambda),L(\gamma))
  = 0$. Now apply $\Hom_G(-,L(\gamma))$ to the short exact sequence $0
  \to L(\mu) \to V(\lambda) \to L(\lambda) \to 0$ to get an exact
  sequence
  \begin{equation*}
    0 = \Ext^1_G(V(\lambda),L(\gamma)) \to \Ext^1_G(L(\mu),L(\gamma) \xrightarrow{\partial}
    \Ext^2_G(L(\lambda),L(\gamma) \to \Ext^2(V(\lambda),L(\gamma)) = 0;
  \end{equation*}
  then $\partial$ is an isomorphism and the claim follows.
\end{proof}

\begin{stmt}
  There is a linear algebraic group $E$ with reductive quotient $\SL_3$, with 
  abelian unipotent radical $W$ of dimension $\dfrac{3}{2}(p-1)(p-2)$, and
  with no Levi decomposition.
\end{stmt}

\begin{proof}
  Write $W = L(\lambda)^\vee \tensor L(\gamma)$; using Steinberg's
  tensor product theorem \cite{JRAG}*{II.3.17}, we observe that $W
  \simeq L(\varpi_2)^{[1]} \tensor L((p-3)\varpi_1)$; the Weyl degree
  formula together with \ref{low-alcove-simple} now shows that $W$ has
  the indicated dimension.  According to
  \ref{non-vanishing-ext2-example}, we may choose a non-zero
  cohomology class
  \begin{equation*}
    [\alpha] \in H^2(G,W) = H^2(G,L(\lambda)^\vee \tensor L(\gamma)) \simeq
    \Ext^2_G(L(\lambda),L(\gamma))
  \end{equation*}
  for $ \alpha \in Z^2(G,W)$; the extension $E=E_\alpha$ of $G$ by $W$
  then has the required properties.
\end{proof}

\subsection{Quasi-split unitary groups in even dimension}
\label{unitary-group-example}
As in \S \ref{parahoric-section}, let $\A$ be a Henselian DVR with
fractions $K$ and residues $k$.  \emph{We suppose in addition that the
  residue field $k$ is algebraically closed and has characteristic $p
  \ne 2$.} In particular, any finite extension of $K$ is totally
ramified over $K$.

We are going to give an example of a non-split reductive group $G$
over $K$ and a parahoric group scheme $\parah$ over $\A$ associated
with $G$ for which the special fiber $\parah_{/k}$ has a Levi
decomposition, but for which there are Levi factors not conjugate by
an element of $\parah(k)$.

Let $L/K$ be a quadratic (hence Galois) extension.  Denote the action
of the non-trivial element $\tau \in \Gamma = \Gal(L/K)$ by $(x
\mapsto x^\tau)$.  Write $\B$ for the integral closure of $\A$.
\begin{stmt}
  \label{dual-numbers}
  $\B \tensor_\A k \simeq k[\e]$ where $k[\e]$ is the $k$-algebra of
  \emph{dual numbers} with basis $1,\e$ for which $\e^2 = 0$.  The
  action of $\tau$ on $\B$ induces the automorphism $\e \mapsto -\e$
  on $k[\e] = \B \tensor_\A k$.
\end{stmt}
\begin{proof}
  Indeed, $\B$ is generated as $\A$-algebra by a uniformizer which
  satisfies an Eisenstein equation of degree 2 over $\A$.
\end{proof}

Let $V$ be a $2n$ dimensional $L$-vector space, and let $h:V \times V
\to L$ be a non-degenerate skew-hermitian form on $V$.  Let $G =
\SU(V,h)$ be the corresponding special unitary group; thus $G$ is a
$K$-form of $\SL_{2n}$.

The fact that $G$ is quasisplit \ref{Gun-quasisplit} means here
that $h$ has maximal index; thus writing $I = \{\pm 1,\pm
2,\dots,\pm n\}$, we may find a basis $\{e_i\mid i \in I\}$
of $V$ such that
 \begin{equation*}
   h(v,w) = \sum_{i = 1}^n (v_i)^\tau w_{-i} - (v_{-i})^\tau w_i
 \end{equation*}
 where we have written e.g. $v = \sum_{i \in I} v_i e_i$ for $v_i \in L$.

 Consider the lattice $\LL = \sum_{i \in I} \B e_i \subset V$ in $V$.
 Note that $h$ determines a skew hermitian form $\LL \times \LL \to \B$.
 Viewing $\LL$ as an $\A$-lattice in the $K$-vector space $V =
 R_{L/K}V$, we may consider the schematic closure $\parah$
 \cite{bruhat-tits-II}*{I.2.6} of $G$ in the $\A$-group scheme
 $\GL(\LL)$; put another way, $\parah$ is the $\A$-structure on $G$
 defined by the $\A$-lattice $\LL$.

\begin{stmt}
  $\parah$ is a parahoric group scheme for $G$.
\end{stmt}

\begin{proof}
  Argue as in \cite{corvallis}*{\S3.11}, where instead the special
  unitary groups for \emph{odd} dimensional $L$-vector spaces are
  treated in detail.
\end{proof}

Write $M = \LL \tensor_\A k$; using \ref{dual-numbers} we view $M$ as a free
$k[\e] = \B \tensor_\A k$-module of rank $2n$.  Write $\beta$ for the
form $\beta = h \tensor 1$; one sees that
\begin{equation*}
  \beta:M \times M \to k[\e]
\end{equation*}
is \emph{skew hermitian} for the involution of $k[\e]$ defined by $\e
\mapsto -\e$.  Write $X \mapsto \sigma(X)$ for the \emph{involution}
(i.e. anti-automorphism of order 2) of $\End_{k[\e]}(M)$ determined by
$\beta$; thus $\beta(Xv,w) = \beta(v,\sigma(X)w)$ for all $v,w \in M$.

The special fiber $\parah_{/k}$ identifies with the closed subgroup
$H=SU(\End_{k[\e]}(M),\sigma)$ of the linear $k$-group $\Aut_{k[\e]}(M)
= \End_{k[\e]}(M)^\times$ defined by 
\begin{equation*}
  H(\Lambda) = \SU(\End_{k[\e]}(M),\sigma)(\Lambda) = \{x \in \End_{k[\e]}(M) \tensor_k \Lambda
  \mid \sigma(x)x = 1, \det(x) = 1\}
\end{equation*}
for each commutative $k$-algebra $\Lambda$.

Since $\parah$ is smooth over $\A$, the group scheme $H = \parah_{/k}$
is smooth over $k$ -- i.e. $H$ is a linear algebraic group
over $k$; of course, the smoothness of $H$ can be checked directly
from its definition.

The bilinear form $\overline{\beta}$ induced on $\overline{M} = M/\e
M$ from $\beta$ by reduction mod $\e$ is non-degenerate and
symplectic.  The involution $\overline{\sigma}$ induced on
$\End_k(\overline{M})$ by $\sigma$ is evidently the adjoint involution
of $\overline{\beta}$.  Consider the natural map $\pi:\Aut_{k[\e]}(M)
\to \GL(\overline{M})$ of linear algebraic groups over $k$.

\begin{stmt}
  \label{image-symplectic}
  $\pi_{\mid H}$ defines a separable surjection $\pi_{\mid H}:H \to
  \SP(\overline{M}) = \SP(\overline{M},\overline{\beta}) \subset
  \GL(\overline{M})$ onto the \emph{symplectic group} of
  $\overline{\beta}$.
\end{stmt}

\begin{proof}
  It is clear that the image of $\pi$ lies in $\SP(\overline{M})$. The
  Lie algebra of $H$  is
  \begin{equation*}
    \Lie(H) = \hlie = \{X \in \End_{k[\e]}(M) \mid 
    \tr(X) = 0, \quad X + \sigma(X) = 0\}.
  \end{equation*}
  where $\tr(X) \in k[\e]$ is the \emph{trace.}  Given $\overline{Y}
  \in \mathfrak{sp}(\overline{M},\overline{\beta})$, let $Y =
  i(\overline{Y}) \in \End_{k[\e]}(M)$ where we have written
  $i:\End_k(\overline{M}) \to \End_{k[\e]}(M)$ for the mapping
  determined by the $k$-linear splitting of the surjection $k[\e] \to
  k$. Then $Y \in \hlie$ so that $d\pi_{\mid \hlie}$ is surjective.
  Since $\SP(\overline{M})$ is connected, $\pi_{\mid H}$ is
  indeed surjective.
\end{proof}

Write
\begin{equation*}
  \Sym(\End_k(\overline{M})) = \Sym(\End_k(\overline{M}),\overline{\sigma}) = 
  \{X \in \End_k(\overline{M}) \mid X = \overline{\sigma}(X)\}
\end{equation*}
for the space of \emph{symmetric} endomorphisms of $\overline{M}$,
and
\begin{equation*}
    \Sym^0(\End_k(\overline{M})) = \{X \in   \Sym(\End_k(\overline{M})) \mid \tr(X) = 0\}.
\end{equation*}
for those of trace 0.  Conjugation makes $\Sym^0(\End_k(\overline{M}))
\subset \Sym(\End_k(\overline{M}))$ into $\SP(\overline{M})$-modules.
\begin{stmt}
  \label{special-fiber-unipotent-radical}
  \begin{enumerate}[(a)]
  \item The kernel of $\pi_{\mid H}$ is
    the unipotent radical $R$ of $H$.
  \item There is an $\SP(\overline{M})$-equivariant isomorphism $\Sym^0(\End_k(\overline{M}))
    \xrightarrow{\sim} R$; in particular, $R$ is a vector group and is
    linearizable for the action of $\SP(\overline{M})$.
  \end{enumerate}
\end{stmt}

\begin{proof}
  Let $R = \ker \pi_{\mid H}$.  Then \ref{image-symplectic} shows that
  $R$ is smooth and that $H/R \simeq \SP(\overline{M})$ is reductive,
  (a) will follow once we know that $R$ is connected and unipotent;
  thus it is enough to prove (b).

  Since $R$ is smooth and $k$ is algebraically closed, it is enough to
  give an isomorphism on $k$-points. A $k$-point $y$ of $\ker \pi_{\mid
    H}$ has the form $y = 1 + \e X$ for $X \in \End_k(\overline{M})$.
  Then $y \in H(k)$  implies
  \begin{equation*}
    1 = y\sigma(y) = (1+\e X)\sigma(1+\e X)  
    = (1+\e X)(1 - \e \overline{\sigma}(X))
    = 1 + \e(X - \overline{\sigma}(X)).
  \end{equation*}
  so that $X = \overline{\sigma}(X)$. Since $1= \det(y) = 1 + \e\tr(X)$
  we deduce that $X \in \Sym^0(\End_k(\overline{M}))$. Thus $X \mapsto
  1 + \e X$ is the required isomorphism.
\end{proof}

\begin{stmt}
  \label{W-cohomology}
  Write $W = \Sym(\End_k(\overline{M}))$ and $W^0 =
  \Sym^0(\End_k(\overline{M}))$. Then
  \begin{enumerate}[(a)]
  \item $H^i(\SP(\overline{M}),W) = 0$ for $i \ge 1$.
  \item $H^i(\SP(\overline{M}),W^0) = 0$ for $i \ge 2$.
  \item $H^1(\SP(\overline{M}),W^0) = \left \{
      \begin{matrix}
        0 & \text{if $n \not \equiv 0 \pmod{p}$,} \\
        k & \text{if $n \equiv 0 \pmod p$.}
      \end{matrix}
      \right .$
    \end{enumerate}
\end{stmt}

\begin{proof}
  The $\SP(\overline{M})$-module $W$ identifies with $\bigwedge^2
  \overline{M}$, which  has a filtration by standard
  modules; cf. \cite{mcninch-dim-crit}*{Lemma 4.8.2} and recall the
  terminology from \S \ref{representation-theory}.  Thus (a) follows
  from \ref{ext-by-Weyl-radical}.

  Moreover, when $p$ does not divide $n$, the result just cited shows
  that $W$ is completely reducible, and that $W^0$ is a simple
  standard module (in fact, $W^0 = V(\varpi_2) = H^0(\varpi_2)$ where
  $\varpi_2$ is a certain fundamental dominant weight). Assertions (b) and (c)
  follow in this case from another application of
  \ref{ext-by-Weyl-radical}.

  Finally, suppose that $p$ divides $n$. Then the identity
  endomorphism of $\overline{M}$ is symmetric and has trace $2n = 0$
  in $k$, so that $W^0$ has a trivial submodule. The result of
  \emph{loc.  cit.} shows that $W^0$ is a Weyl module (again, $W^0 =
  V(\varpi_2)$) which has composition length 2, and there is a short
  exact sequence $0 \to W^0 \to W \to k \to 0$ of
  $\SP(\overline{M})$-modules. The resulting long exact sequence
  together with (a) now gives (b) and (c) in this case.
\end{proof}

\begin{prop}
  \begin{enumerate}[(a)]
  \item $H = \parah_{/k}$ has a Levi factor.
  \item If $n \not \equiv 0 \pmod p$, and if $L,L'$ are Levi factors
    of $H$, then $L = gL'g^{-1}$ for some $g \in H(k)$.
  \item If $n \equiv 0 \pmod p$, there are Levi factors $L,L'$ of $H =
    \parah_{/k}$ which are not conjugate by an element of $H(k)$.
  \end{enumerate}
\end{prop}

\begin{proof}
  In view of \ref{special-fiber-unipotent-radical} and
  \ref{W-cohomology}, assertion (a) follows from Theorem
  \ref{easy-cohomology-existence}, assertion (b) follows from Theorem
  \ref{levi-conjugacy}, and assertion (c) follows from
  \ref{conj-of-complements}.
\end{proof}

\newcommand\mylabel[1]{#1\hfil}


\begin{bibsection}
  \begin{biblist}[\renewcommand{\makelabel}{\mylabel} \resetbiblist{XXXXX}]   

  \bib{borel-LAG}{book}{
      author = {Borel, Armand},
      title = {Linear Algebraic Groups},
      year = {1991},
      publisher = {Springer Verlag},
      volume = {126},
      series = {Grad. Texts in Math.},
      edition = {2nd ed.},
      label = {Bo 91}
    }

    \bib{borel-springer}{article}{ 
      author={Borel, A.},
      author={Springer, T.~A.}, 
      title={Rationality properties of linear
        algebraic groups. II}, 
      date={1968}, 
      journal={T\^ohoku Math. J.
        (2)}, 
      volume={20}, 
      pages={443\ndash 497}, 
      label = {BS68}}

     \bib{borel-tits}{article}{
       author = {Borel, Armand},
       author = {Tits, Jacques},
       title = {Groupes R\'eductifs},
       journal = {Publ. Math. IHES},
       volume = {27},
       year = {1965},
       pages = {55--151},
       label = {BoTi 65}}

     \bib{bourbaki}{book}{
       author={Bourbaki, Nicolas},
       title={Lie groups and Lie algebras. Chapters 4--6},
       series={Elements of Mathematics (Berlin)},
       note={Translated from the 1968 French original by Andrew Pressley},
       publisher={Springer-Verlag},
       place={Berlin},
       date={2002},
       label = {Bou 02}}



     \bib{Brown}{book}{
       author={Brown, Kenneth S.},
       title={Cohomology of groups},
       series={Graduate Texts in Mathematics},
       volume={87},
       note={Corrected reprint of the 1982 original},
       publisher={Springer-Verlag},
       place={New York},
       date={1994},
       pages={x+306},
       label={Br 94}
     }
     



     \bib{bruhat-tits-I}{article}{
       author={Bruhat, F.},
       author={Tits, J.},
       title={Groupes r\'eductifs sur un corps local},
       language={French},
       journal={Inst. Hautes \'Etudes Sci. Publ. Math.},
       number={41},
       date={1972},
       pages={5--251},
       label = {BrTi 72}}
     
     \bib{bruhat-tits-II}{article}{
       author={Bruhat, F.},
       author={Tits, J.},
       title={Groupes r\'eductifs sur un corps local. II. Sch\'emas en groupes.
         Existence d'une donn\'ee radicielle valu\'ee},
       language={French},
       journal={Inst. Hautes \'Etudes Sci. Publ. Math.},
       number={60},
       date={1984},
       pages={197--376},
       label = {BrTi 84}}
     
     \bib{conrad-gabber-prasad}{book}{
        author = {Brian Conrad},
        author = {Ofer Gabber},
        author = {Gopal Prasad},
        title = {Pseudo-reductive groups},
        publisher = {Cambridge University Press},
        series = {New Mathematical Monographs},
        volume = {17},
        label = {CGP 10},
        year = {2010}
      }

      \bib{demazure-gabriel}{book}{
        author={Demazure, Michel},
        author={Gabriel, Pierre},
        title={Groupes alg\'ebriques. Tome I: G\'eom\'etrie alg\'ebrique,
          g\'en\'eralit\'es, groupes commutatifs},
        language={French},
        publisher={Masson \& Cie, \'Editeur, Paris},
        date={1970},
        pages={xxvi+700},
        label={DeGa 70}
      }


      \bib{humphreys}{article}{
        author={Humphreys, J. E.},
        title={Existence of Levi factors in certain algebraic groups},
        journal={Pacific J. Math.},
        volume={23},
        date={1967},
        pages={543--546},
        label={Hu 67}
      }


      \bib{Jacobson}{book}{
        author={Jacobson, Nathan},
        title={Lie algebras},
        note={Republication of the 1962 original},
        publisher={Dover Publications Inc.},
        place={New York},
        date={1979},
        pages={ix+331},
        label = {Jac 79}}
      
      \bib{JRAG}{book}{
        author={Jantzen, Jens Carsten},
        title={Representations of algebraic groups},
        series={Mathematical Surveys and Monographs},
        volume={107},
        edition={2},
        publisher={American Mathematical Society},
        place={Providence, RI},
        date={2003},
        pages={xiv+576},
        label={Ja 03}        
      }
		
      \bib{jantzen-semisimple}{article}{
        author={Jantzen, Jens Carsten},
        title={Low-dimensional representations of reductive groups are
          semisimple},
        conference={
          title={Algebraic groups and Lie groups},
        },
        book={
          series={Austral. Math. Soc. Lect. Ser.},
          volume={9},
          publisher={Cambridge Univ. Press},
          place={Cambridge},
        },
        date={1997},
        pages={255--266},
        label = {Ja 97}    }

        \bib{mcninch-witt}{article}{
          author={McNinch, George J.},
          title={Faithful representations of $\rm SL_2$ over truncated Witt
            vectors},
          journal={J. Algebra},
          volume={265},
          date={2003},
          number={2},
          pages={606--618},
          label = {Mc 03}
        }

        \bib{mcninch-dim-crit}{article}{
          author={McNinch, George J.},
          title={Dimensional criteria for semisimplicity of representations},
          journal={Proc. London Math. Soc. (3)},
          volume={76},
          date={1998},
          number={1},
          pages={95--149},
          label={Mc 98}
        }
	

        \bib{oesterle}{article}{
          author={Oesterl{\'e}, Joseph},
          title={Nombres de Tamagawa et groupes unipotents en caract\'eristique
            $p$},
          language={French},
          journal={Invent. Math.},
          volume={78},
          date={1984},
          number={1},
          pages={13--88},
          label = {Oe 84}
        }


         \bib{prasad-galois-fixed}{article}{
           author={Prasad, Gopal},
           title={Galois-fixed points in the Bruhat-Tits building of a reductive
             group},
           language={English, with English and French summaries},
           journal={Bull. Soc. Math. France},
           volume={129},
           date={2001},
           number={2},
           pages={169--174},
           label = {Pr 01}
         }


        \bib{Prasad-Raghunathan}{article}{
          author={Prasad, Gopal},
          author={Raghunathan, M. S.},
          title={Topological central extensions of semisimple groups over local
            fields},
          journal={Ann. of Math. (2)},
          volume={119},
          date={1984},
          number={1},
          pages={143--201},
          label = {PR 84}
        }

        
        \bib{rosenlicht-rationality}{article}{
          author={Rosenlicht, Maxwell},
          title={Questions of rationality for solvable algebraic groups over
            nonperfect fields},
          language={English, with Italian summary},
          journal={Ann. Mat. Pura Appl. (4)},
          volume={61},
          date={1963},
          pages={97--120},
          label = {Ro 63}
        }
        
         \bib{rousseau}{article}{
           author = {Rousseau, G.},
           title  ={Immeubles des groupes r\'eductifs sur les corps locaux},
           journal = {Th\`ese, Universit\'e de Paris-Sud},
           year = {1977},
           label = {Ro 77}
           }


                
        \bib{Serre-LocalFields}{book}{
          author={Serre, Jean-Pierre},
          title={Local fields},
          series={Graduate Texts in Mathematics},
          volume={67},
          note={Translated from the French by Marvin Jay Greenberg},
          publisher={Springer-Verlag},
          place={New York},
          date={1979},
          label={Se 79}}

    
        \bib{SerreGC}{book}{
          author={Serre, Jean-Pierre},
          title={Galois cohomology},
          note={Translated from the French by Patrick Ion and revised by the
            author},
          publisher={Springer-Verlag},
          place={Berlin},
          date={1997},
          pages={x+210},
          review={MR 98g:12007},
          label = {Ser97}}
        
        \bib{serre-AGCF}{book}{
          author={Serre, Jean-Pierre},
          title={Algebraic groups and class fields},
          series={Graduate Texts in Mathematics},
          volume={117},
          note={Translated from the French},
          publisher={Springer-Verlag},
          place={New York},
          date={1988},
          pages={x+207},
          label = {Se 88}
        }

        \bib{springer-LAG}{book}{ 
          author={Springer, Tonny~A.}, 
          title={Linear algebraic groups}, 
          edition={2}, 
          series={Progr. in Math.},
          publisher={Birkh{\"a}user}, address={Boston}, date={1998},
          volume={9}, 
          label={Sp 98}}


        \bib{corvallis}{article}{
          author={Tits, J.},
          title={Reductive groups over local fields},
          conference={
            title={Automorphic forms, representations and $L$-functions (Proc.
              Sympos. Pure Math., Oregon State Univ., Corvallis, Ore., 1977), Part
              1},
          },
          book={
            series={Proc. Sympos. Pure Math., XXXIII},
            publisher={Amer. Math. Soc.},
            place={Providence, R.I.},
          },
          date={1979},
          pages={29--69},
          label = {Ti 77}
        }
	
        \bib{weibel}{book}{
          author={Weibel, Charles A.},
          title={An introduction to homological algebra},
          series={Cambridge Studies in Advanced Mathematics},
          volume={38},
          publisher={Cambridge University Press},
          place={Cambridge},
          date={1994},
          pages={xiv+450},
          label = {We 94}} 
        

  \end{biblist}
\end{bibsection}
\end{document}